\documentclass[12pt]{article} 
\usepackage{amsmath}
\usepackage{amssymb}
\usepackage[utf8]{inputenc}
\usepackage[T1]{fontenc}
\usepackage{amsthm}
\usepackage[top=1in, bottom=1in, left=1in, right=1in]{geometry}
\usepackage{xcolor}
\usepackage{hyperref}
\usepackage{comment}
\usepackage{tikz-cd}
\usepackage{enumerate}

\DeclareUnicodeCharacter{0301}{\'{}}

\usepackage{mathtools}
\mathtoolsset{showonlyrefs}

\newtheorem{theorem}{Theorem}[section]
\newtheorem{definition}[theorem]{Definition}

\newtheorem{lemma}[theorem]{Lemma}

\newtheorem{proposition}[theorem]{Proposition}
\newtheorem{corollary}[theorem]{Corollary}
\newtheorem{remark}[theorem]{Remark}

\newcommand{\PP}{{\mathbb P}} \newcommand{\RR}{\mathbb{R}}
\newcommand{\QQ}{\mathbb{Q}} \newcommand{\CC}{{\mathbb C}}

\newcommand{\cX}{\mathcal{X}}

\newcommand{\cE}{\mathcal{E}}

\newcommand{\cH}{\mathcal{H}}

\newcommand{\cJ}{\mathcal{J}}

\newcommand{\cL}{\mathcal{L}}

\newcommand{\cR}{\mathcal{R}}

\newcommand{\Ric}{\operatorname{Ric}}
\newcommand{\vol}{\operatorname{vol}}
\newcommand{\ord}{\mathrm{ord}}

\newcommand{\ddc}{\mathrm{dd^c}}

\makeatletter
\renewcommand\tableofcontents{%
  \section*{\contentsname
    \@mkboth{\MakeUppercase\contentsname}{\MakeUppercase\contentsname}}%
  \vspace{0.5em}% space after "Contents"
  \@starttoc{toc}%
}
\renewcommand\l@section[2]{%
  \vspace{0.5em}% extra space between section lines
  \@dottedtocline{1}{0em}{2.0em}{\normalsize\bfseries #1}{\normalsize\bfseries #2}}
\makeatother
\title{Understanding entropy using Aubin--Yau deformations}
\author{Tam\'as Darvas and Kewei Zhang}
\date{}

\begin{document}

\maketitle

\begin{abstract}
We investigate three types of entropy that frequently appear in the research of K\"ahler geometry: the classical K\"ahler entropy, the radial entropy and the non-Archimedean entropy. We give a unifying interpretation for these entropies by showing that they can all be expressed as the derivative of the Monge--Amp\`ere energy along certain Aubin--Yau type deformations. 
Such considerations also lead to a novel algebraic criterion for the existence of constant scalar curvature K\"ahler metrics.
\end{abstract}

\setcounter{tocdepth}{1}

\tableofcontents

\section{Introduction}

In geometric analysis, entropy has emerged as a fundamental concept that permeates a remarkably broad spectrum of problems. Its influence extends from metric geometry, where entropy measures the asymptotic growth of volumes and encodes curvature information of spaces \cite{Oh14}, to the study of geometric partial differential equations, where entropy functionals provide essential monotonicity formulas and compactness tools, as exemplified by Perelman's work on the Ricci flow \cite{Pe02} and also in more recent works such as \cite{GPSS24,GT25}. 

In the study of canonical metrics, entropy also plays a central role: it appears as the most delicate term in the Mabuchi K energy, whose critical points are precisely the constant scalar curvature Kähler (cscK) metrics. It also lies at the heart of the variational approach to canonical metrics and the Yau--Tian--Donaldson conjecture \cite{BBJ18,Li20-cscK,BJ25preprint,DZ25,MPWN}. Moreover, entropy has become indispensable in the algebro-geometric theory of K-stability, providing analytic insights that bridge differential geometry and algebraic geometry through the non-Archimedean framework \cite{BHJ17,Li21,Pic24,T26}. The purpose of this paper is to explain a useful way of viewing entropy: it is the variation of the Monge-Ampère energy along solutions of Aubin-Yau type problems. To this end, we will mainly focus on three types of entropy that frequently appear in the literature on K\"ahler geometry.

\subsection{K\"ahler entropy}
Let $(X,\omega)$ be a compact K\"ahler manifold of dimension $n$. Let
$$
\cH:=\{u\in C^{\infty}(X,\RR):\omega_u:=\omega+\ddc u>0\}
$$
be the space of smooth K\"ahler potentials and $\cE^1$ be its metric completion with respect to the $d_1$ metric \cite{Dar15}. Here $\ddc=i\partial\bar\partial/2\pi$.

Set $V:=\int_X\omega^n$. The standard K\"ahler entropy of $u\in\cE^1$ is defined as
\begin{equation}
    \label{eq: def_Ent}
    \mathrm{Ent}(u):=\frac{1}{V}\int_X\log\frac{\omega^n_u}{\omega^n}\omega^n_u
\end{equation}
when $\omega^n_u$ is absolutely continuous with respect to $\omega^n$, and $\mathrm{Ent}(u):=+\infty$ otherwise. 

This entropy is the most nonlinear part of the Mabuchi functional, and it is
also the term that is most difficult to control by approximation.  
In \cite{DZ25}, motivated by the study of cscK metrics, we proposed a way to regularize the entropy using a family of Aubin--Yau type equations, which we now describe.

For $u\in\mathcal E^1$ and $\beta>0$, following the terminology of \cite{DZ25}, let
$u^\beta\in\mathcal E^1\cap C^0(X)$ denote the unique solution of the Aubin--Yau equation
\begin{equation}\label{eq: intro_AY}
\omega_{u^\beta}^n=e^{\beta(u^\beta-u)}\omega^n.
\end{equation}
As pointed out in \cite{DZ25}, there is intriguing asymptotic behavior as $\beta\to\infty$. In view of this, following a suggestion of Berman \cite{Ber19}, $u^\beta$ is called in \cite{DZ25} the transcendental quantization of $u$ at level $\beta$.

As will soon become clear, however, shifting the asymptotic analysis from $\infty$ to $0$ has certain advantages. Therefore, in this work it will often be more convenient to use the parameter $\varepsilon:=1/\beta>0$ and adopt the somewhat redundant but useful convention
\[
u_\varepsilon:=u^{1/\varepsilon}=u^\beta.
\]
We then have
\begin{equation}\label{eq: def_u_e}
\omega_{u_\varepsilon}^n=e^{\frac{1}{\varepsilon}(u_\varepsilon-u)}\omega^n.
\end{equation}
In particular,
\begin{equation}\label{eq: ent_ue_int}
\mathrm{Ent}(u_\varepsilon)
=\frac{1}{\varepsilon V}\int_X(u_\varepsilon-u)\omega^n_{u_\varepsilon}.
\end{equation}

Let
\begin{equation}
    \label{eq: def_E}
    I(u):=\frac{1}{(n+1)V}\int_X \sum_{i=0}^n u\omega^i_u\wedge\omega^{n-i}
\end{equation}
denote the Monge--Amp\`ere energy of $u\in\cE^1$.
As shown in \cite[Proposition 3.6]{DZ25} (cf. also \cite{BBGZ}), the solution $u_\varepsilon$ of \eqref{eq: def_u_e}, is the unique normalized maximizer of Ding type functional $F^\beta(\cdot,u)$:
\begin{equation}
    \label{eq:sup_Ding}
    \sup_{v\in\cE^1\cap C^0(X)}F^\beta(v,u) :=  \sup_{v\in\cE^1\cap C^0(X)}\left(-\log\int_X e^{\frac{1}{\varepsilon}(v-u)}\frac{\omega^n}{V}+\frac{1}{\varepsilon}(I(v)-I(u))\right)
\end{equation}
that satisfies the normalization condition
$
\int_X e^{\frac{1}{\varepsilon}(v-u)}\frac{\omega^n}{V}=1.
$

Now, for any fixed $u\in\cE^1$, let us consider the following function:
\begin{equation*}
f_u(\varepsilon):=
    \begin{cases}
        I(u_\varepsilon),\ \varepsilon>0,\\
        I(u),\ \varepsilon=0.
        
    \end{cases}
\end{equation*}
Our first main result shows that $\mathrm{Ent}(\cdot)$ is the derivative of $I(\cdot)$ along the Aubin--Yau path \eqref{eq: def_u_e}.

\begin{theorem}[K\"ahler Aubin--Yau deformations]
    \label{thm:main_f'}
    The function $[0,+\infty)\ni\varepsilon\mapsto f_u(\varepsilon)$ is concave, non-decreasing and belongs to  $C^1([0,\infty))$. Moreover,
    $$
    f_{u}'(\varepsilon)=\mathrm{Ent}(u_\varepsilon)\text{ for any $\varepsilon>0$ and }
    f'_{u,+}(0)=\mathrm{Ent}(u).
    $$
\end{theorem}

In fact Lemma \ref{lem: Ent_Holder} implies that $f$ is also in $C^{1,1/4^n}(0,\infty)$.
Thus the entropy of $u$ is recovered as the initial velocity of the energy of
its Aubin--Yau deformations.  

The proof of Theorem \ref{thm:main_f'} is rather straightforward, which builds on the analysis of the equation \eqref{eq: def_u_e} in \cite[\S 3]{DZ25}. We first prove the result for smooth $u\in\cH$ and then treat general $u\in\cE^1$ by approximation. Moreover, from this proof it follows easily that the entropy $\mathrm{Ent}(u)$ for any $u\in\cE^1$ can be approximated by a smooth sequence $v_i\in\cH$ such that $v_i\xrightarrow{d_1}u$ and $\mathrm{Ent}(v_i)\to \mathrm{Ent}(u)$, reproducing \cite[Theorem 3.2]{BDL17}.

\subsection{Radial entropy}

We next pass from potentials to geodesic rays.  Let $\mathcal R^1$ denote the
space of finite energy geodesic rays in $\cE^1$ emanating from the origin, and, for a ray
$\{u_t\}\in\mathcal R^1$, set
\[
 I\{u_t\}:=\lim_{t\to\infty}\frac{I(u_t)}{t},\qquad
 \mathrm{Ent}\{u_t\}:=\lim_{t\to\infty}\frac{\mathrm{Ent}(u_t)}{t}
\]
whenever the latter entropy slope is finite, and $\infty$ otherwise. We refer to \cite{DL19} for a detailed study of the space $\mathcal R^1$.

In \cite[\S 6]{DZ25}, given any $\beta>0$ and $\{u_t\}\in\cR^1$, we introduced a radial
Ding type functional $F^\beta(\{\cdot\},\{u_t\})$: 
$$
F^\beta(\{v_t\},\{u_t\}):=L^\beta(\{v_t\},\{u_t\})+\beta(I\{v_t\}-I\{u_t\}),\ \{v_t\}\in\cR^1,
$$
where
$$
L^\beta(\{v_t\},\{u_t\}):=\liminf_{t\to\infty}\left(\frac{-1}{t}\log\int_Xe^{\beta(v_t-u_t)}\frac{\omega^n}{V}\right)
$$
We will construct for each ray $\{u_t\}$ and $\beta>0$ a distinguished radial Aubin--Yau deformation
\[
 \{P(u)_t^\beta\}\in\mathcal R^1.
\]
We refer to \eqref{eq: P_beta_u_def} for the precise definition, where $\{P(u)_t^\beta\}$ is described as the supremum/envelope of finite energy subgeodesic rays $\{v_t\}$ satisfying the normalization
\[
L^\beta(\{v_t\},\{u_t\})\geq 0.
\]
Most importantly, Theorem \ref{thm: maximizer_main} shows that $\{P(u)_t^\beta\}$ is in fact a bounded geodesic ray and is the unique maximizer $\{v_t\}$ of $F^\beta(\{\cdot\},\{u_t\})$ satisfying
\[
L^\beta(\{v_t\},\{u_t\})=0.
\]

Thus, by analogy with \eqref{eq:sup_Ding}, the family $\{P(u)_t^\beta\}$ can be viewed as the solution of a certain radial Aubin--Yau type equation, even though we cannot write down an explicit equation satisfied by $\{P(u)_t^\beta\}$!

As in previous discussions, we will interchangeably use the parameter $\varepsilon=1/\beta$ and write
\[
 P_\varepsilon(u)_t:=P(u)_t^{1/\varepsilon}
 \quad\text{and}\quad P_0(u)_t:=u_t.
\]
The radial analogue of Theorem \ref{thm:main_f'} is the
following.

\begin{theorem}[Radial Aubin--Yau deformation]\label{thm:intro_radial}
Let $\{u_t\}\in\mathcal R^1$ with $\mathrm{Ent}\{u_t\}<\infty$. Define
\begin{equation} \label{eq: H_u_def}
 H_u(\varepsilon):=
 \begin{cases}
 I\{P_\varepsilon(u)_t\},&\varepsilon>0,\\
 I\{u_t\},&\varepsilon=0.
 \end{cases}
\end{equation}
Then $H_u$ is concave, non-decreasing and belongs to $C^{1}([0,+\infty)) \cap C^{1,1/4^n}((0,+\infty))$.
Moreover,
\[
 H_u'(\varepsilon)=\mathrm{Ent}\{P_\varepsilon(u)_t\}\text{ for }\varepsilon>0\text{ and }
 H_{u,+}'(0)=\mathrm{Ent}\{u_t\}.
\]
\end{theorem}

Theorem \ref{thm:intro_radial} is obtained by taking the asymptotic slope at infinity in Theorem \ref{thm:main_f'} along a geodesic ray. From this perspective, the K\"ahler and radial formulas arise as two instances of the same underlying variational principle. More precisely, given a fixed ray $\{u_t\}$, one may either solve the K\"ahler Aubin--Yau problem at each time $t$, obtaining $u_{t,\varepsilon}$, and subsequently pass to the limit $t\to\infty$, or pass directly to the radial setting and solve the corresponding radial Aubin--Yau deformation problem, obtaining $\{P_\varepsilon(u)_t\}$. While the rays $\{u_{t,\varepsilon}\}$ and $\{P_\varepsilon(u)_t\}$ need not coincide, they determine the same asymptotic energy, and the asymptotic slope of the K\"ahler entropy coincides with the radial entropy. This is the content of our next main result:

\begin{theorem}[Slope compatibility]\label{thm: intro_slope}
Let $\{u_t\}\in\mathcal R^1$, and for $t>0$ and $\varepsilon>0$ set
\[
 g_t(\varepsilon):=\frac{I(u_{t,\varepsilon})}{t},
 \qquad
 g_\infty(\varepsilon):=I\{P_\varepsilon(u)_t\}.
\]
Then $g_t(\cdot)\to g_\infty(\cdot)$ locally uniformly on $(0,+\infty)$, as $t \to \infty$, and the derivatives
converge locally uniformly as well. In particular, for every $\varepsilon>0$,
\[
 \lim_{t\to\infty}\frac{I(u_{t,\varepsilon})}{t}
 =I\{P_\varepsilon(u)_t\},
 \qquad
 \lim_{t\to\infty}\frac{\mathrm{Ent}(u_{t,\varepsilon})}{t}
 =\mathrm{Ent}\{P_\varepsilon(u)_t\}.
\]
\end{theorem}

The critical part of the proof of the above result is contained in Proposition \ref{prop: thm12cor}, that provides quantitative continuity estimates for the Monge--Ampère energy of Aubin--Yau solutions with respect to both the potential and the parameter $\beta=1/\varepsilon$.  In the radial setting the main additional
point is to identify the derivative of $g_\infty$ with radial entropy.  This is
first done for rays induced by test configurations, where the statement follows
from intersection-theoretic slope formulas in \cite{Li21}, and then extended to arbitrary
finite energy rays using Demailly approximation and lower semicontinuity of
radial entropy.

\subsection{Non-Archimedean entropy}
\label{subset: NA}
An analogous Aubin--Yau variational picture emerges on the non-Archimedean (NA) context as well. Since Aubin--Yau type equations have not yet been given a precise formulation in NA geometry, we propose a variational definition of the corresponding ``NA Aubin--Yau objects,'' motivated by the radial setting discussed above.

Let $(X,L)$ be a complex projective manifold, with $L$ being an ample line bundle on $X$. Let us recall some standard notions from NA geometry. We refer the reader to \cite{BoJ18a,BoJ18,BBJ18,BJ22b} for more background.

Let $X^{an}$ denote the Berkovich analytification of $X$ with respect to the trivial absolute value on the ground field $\CC$. We equip $X^{an}$ with the Berkovich topology, so that it is a compact topological space. Inside $X^{an}$ there is a dense subset $X^{div}$ consisting of divisorial valuations on $X$.
Each $v\in X^{div}$ is of the form $v=\lambda\ord_F$, where $\lambda\in\QQ_{\geq 0}$ and $F\subset Y\xrightarrow{\pi}X$ is a prime divisor on some smooth birational model over $X$.
One then defines the log discrepancy function $A_X: X^{div}\to [0,+\infty)$ as
$$
A_X(v)=A_X(\lambda\ord_F):=\lambda(1+\ord_F(K_Y-\pi^*K_X)).
$$
When $\lambda=0$, $v=v_{triv}$ is the trivial valuation, so that $A_X(v_{triv})=0$; otherwise, $A_X(v)>0$.
And $A_X(\cdot)$ admits a maximal lower semi-continuous (lsc) extension to $X^{an}$, so that $A_X:X^{an}\to [0,+\infty]$ is lsc on $X^{an}$ (see \cite[Appendix A]{BoJ18}).

Following \cite[\S 6.2]{BBJ18}, we denote by $\cH^{\mathrm{NA}}$ the space of continuous $L$-psh functions on $X^{an}$ arising from test configurations of $(X,L)$ (see \S\ref{sec: pre}). This space is the NA analogue of $\cH$ for $(X,\omega)$ (where $\omega\in c_1(L)$). A general $L$-psh function on $X^{an}$ is not identically $-\infty$ and is the decreasing limit of a net of functions in $\cH^{\mathrm{NA}}$. We denote by $\mathrm{PSH}^{\mathrm{NA}}$ the set of all such functions, and each $U\in\mathrm{PSH}^{\mathrm{NA}}$ is upper semi-continuous (usc) on $X^{an}$. One can also define the NA analogue of $\cE^1$, denoted by $\cE^{1,\mathrm{NA}}$, which consists of $L$-psh functions with finite NA Monge--Amp\`ere energies (see \S\ref{sub: NA_pre}).

For any $U\in\cE^{1,\mathrm{NA}}$, one is also interested in the NA entropy of $U$, which is defined as (see \cite[\S 2.4]{BoJ18})
$$
\mathrm{Ent}^{\mathrm{NA}}(U):=\int_{X^{an}}A_X\ \mathrm{MA}^{\mathrm{NA}}(U).
$$
Here $\mathrm{MA}^{\mathrm{NA}}(U)$ denotes the NA Monge--Amp\`ere measure of $U$ (as defined in \cite[Theorem 6.9]{BoJ18a}), which is a probability Radon measure on $X^{an}$. 

%Note that $\cE^{1,\mathrm{NA}}$ is equipped with some strong topology (see \cite[\S 6.6]{BoJ18a}). For us, we do not need to work with this topology specifically, but only use one aspect of it: if there is a decreasing (or a uniformly convergent) sequence $U^k\to U$ in $\cE^{1,\mathrm{NA}}$, then one has the weak convergence $\mathrm{MA}^{\mathrm{NA}}(U^k)\to \mathrm{MA}^{\mathrm{NA}}(U)$ of probability measures (see \cite[\S 6.2]{BoJ18a}). So in particular, $\mathrm{Ent}^{\mathrm{NA}}:\cE^{1,\mathrm{NA}}\to[0,+\infty]$ is lsc along such a sequence.

Now as in the K\"ahler case, for any $U\in\cE^{1,\mathrm{NA}}$ and $\varepsilon>0$, we would like to define some continuous potential $U_\varepsilon\in \cE^{1,\mathrm{NA}}$ that is the solution to some NA version of the Aubin--Yau type equation \eqref{eq: def_u_e}. Unfortunately, it seems that such NA equations do not quite make sense on $X^{an}$ (due to the fact that there is no canonical background NA measure associated to the pair $(X,L)$). However, it remains meaningful to seek a maximizer of the following NA Ding-type quantity (compare \eqref{eq:sup_Ding}):
\begin{equation}
    \label{eq:sup_NA_Ding}
    \sup_{V\in \cE^{1,\mathrm{NA}}\cap C^0(X^{an})}\left(\mathrm{L}^{\mathrm{NA}}\left(\frac{1}{\varepsilon}(U-V)\right)+\frac{1}{\varepsilon}(I^{\mathrm{NA}}(V)-I^{\mathrm{NA}}(U))\right).
\end{equation}
Here $\mathrm{L}^{\mathrm{NA}}(f):=\inf_{X^{div}}(A_X+f)$ is the NA $L$ energy for some (not necessarily continuous) function $f$ on $X^{an}$ (cf. \cite[\S 2.6]{BoJ18}) and $I^{\mathrm{NA}}$ is the NA Monge--Amp\`ere energy (see \S\ref{sec: pre}). The quantity \eqref{eq:sup_NA_Ding} arises naturally from the NA $\beta$-entropy considered in \cite[\S 8]{DZ25}.

We show in Corollary \ref{cor:U_e_is_max} that for any $\varepsilon>0$ and $U\in\cE^{1,\mathrm{NA}}$ there exists a unique normalized maximizer $U_\varepsilon\in \cE^{1,\mathrm{NA}}\cap C^0(X^{an})$ for \eqref{eq:sup_NA_Ding}, satisfying
$$
\mathrm{L}^{\mathrm{NA}}\left(\frac{1}{\varepsilon}(U-U_\varepsilon)\right)=0.
$$
In this sense, $U_\varepsilon$ is indeed the NA analogue for the solution $u_\varepsilon$ to the Aubin--Yau type equation \eqref{eq: def_u_e}.

In contrast to the K\"ahler setting, but similarly to the radial one, the maximizer $U_\varepsilon$ admits a rather explicit description as an envelope (see Corollary \ref{cor:U_e_is_max}).
$$
U_\varepsilon=\mathrm{P}(U+\varepsilon A_X)=\sup\{V: V\in \mathrm{PSH}^{\mathrm{NA}},\ V\leq U+\varepsilon A_X\}.
$$
To the authors' knowledge, such an envelope first appeared in Xia's work \cite[\S7]{X23}.
In Proposition \ref{prop: U_e_conti} we will prove a somewhat surprising result, saying that the above envelope is always a continuous $L$-psh function on $X^{an}$ for any $U\in\cE^{1,\mathrm{NA}}$ and $\varepsilon>0$. Moreover, one has $U_\varepsilon\searrow U$ as $\varepsilon\searrow 0$ (see Corollary \ref{cor:Ue_to_U}).

As in the K\"ahler setting, we consider
$
I^{\mathrm{NA}}(U_\varepsilon)\text{ for }\varepsilon\in [0,+\infty).
$
Now we state the NA version of Theorem \ref{thm:main_f'}.

\begin{theorem}[NA Aubin-Yau deformations]
    \label{thm: main_F'}
    For any $U\in\cE^{1,\mathrm{NA}}$, the function $[0,+\infty)\ni\varepsilon\mapsto I^{\mathrm{NA}}(U_\varepsilon)$ is concave, non-decreasing and belongs to  $C^1([0,\infty))$. Moreover,
\begin{equation}\label{eq: NA_id}
    \frac{d}{d\varepsilon}I^{\mathrm{NA}}(U_\varepsilon)=\mathrm{Ent}^{\mathrm{NA}}(U_\varepsilon)\text{ for any $\varepsilon>0$ and }
    \frac{d}{d\varepsilon}\bigg|_{\varepsilon=0^+}I^{\mathrm{NA}}(U_\varepsilon)=\mathrm{Ent}^{\mathrm{NA}}(U).
\end{equation}
\end{theorem}

The proof follows a similar strategy as the one of Theorem \ref{thm:main_f'}, but relies also heavily on the deep NA machinery developed in \cite{BoJ18a,BoJ18,BBJ18,BJ22b}. Additionally, the second identity of this theorem also confirms a conjecture of Xia \cite[Remark 7.1]{X23}. 

The above result readily implies that one can regularize the NA entropy of $U$ using the decreasing continuous family $U_\varepsilon$. As shown in \cite[Theorem B]{BoJ18}, any continuous $L$-psh function arises from some filtration of the section ring $R(X,L)$. So the NA entropy of any $U\in\cE^{1,\mathrm{NA}}$ can be approximated using a sequence of filtrations, thus reproducing a slightly weaker version of Li's regularization result \cite[Proposition 6.3]{Li20-cscK}. Recall that Li used the NA Calabi--Yau theorem \cite{BFJ15a,BoJ18a} to prove his result, while our work shows that one can alternatively use Aubin--Yau type equations to carry out the regularization procedure. Actually, from the proof of Theorem \ref{thm: main_F'} one can furthermore deduce the following more refined regularization result for the NA entropy:

\begin{theorem}
\label{thm: main_reg_NA_ent}
    For any $U\in\cE^{1,\mathrm{NA}}$, there exists a sequence $U^k\in\cH^{\mathrm{NA}}$ with $U^k\searrow U$ and a sequence of rational numbers $\varepsilon_k\searrow 0$ such that $U^k_{\varepsilon_k}\searrow U$ and
    $$
    \mathrm{Ent}^{\mathrm{NA}}(U)=\lim_{k\to\infty}\mathrm{Ent}^{\mathrm{NA}}(U^k_{\varepsilon_k}).
    $$
\end{theorem}

As shown in \cite[\S 9]{DZ25}, potentials of the form $U^{k}_{\varepsilon_k}$ correspond to log discrepancy models of $(X,L)$ (introduced in \cite[(11)]{DZ25}), so this yields an alternative (and more conceptual) proof of our regularization result \cite[Theorem 1.5]{DZ25}. We refer to \cite{T26} for a resolution of the Boucksom--Jonsson conjecture, showing that it is possible to approximate finite energy non-Archimedean entropy using non-Archimedean entropy of semi-ample test configurations.

\paragraph{The three approaches unified.} We will see in Theorem \ref{thm: 3kingdoms} that $I^{\mathrm{NA}}(U_\varepsilon)$ can be computed as the Monge--Amp\`ere slope of a distinguished (non-geodesic) ray, thereby bringing together the K\"ahler, radial, and non-Archimedean approaches:
\begin{equation}
\label{eq: NA_E=rad_E}
I^{\mathrm{NA}}(U_\varepsilon)=I\{P_\varepsilon(u)_t\}=I\{u_{t,\varepsilon}\}:
=\liminf_{t\to\infty}\frac{I(u_{t,\varepsilon})}{t},
\qquad \varepsilon\geq 0.
\end{equation}
Here $\{u_t\}$ is the maximal geodesic ray associated to $U\in\cE^{1,\mathrm{NA}}$ (see \cite[\S 6.4]{BBJ18}) and for each $t>0$ the potential $u_{t,\varepsilon}$ is the solution of \eqref{eq: def_u_e} along $u_t$ (when $\varepsilon=0$ we simply define $u_{t,0}=u_t$). When $U\in\cH^{\mathrm{NA}}$, the slope formula \eqref{eq: NA_E=rad_E} has already been established in \cite[\S9]{DZ25}. Here we are proving that the same identity holds for more general $U\in\cE^{1,\mathrm{NA}}$ as well. In fact, building on a more refined non-Archimedean analysis, we establish a stronger slope formula that applies even to possibly non-maximal geodesic rays (see Theorem \ref{thm: rad_E=sup}).

\subsection{Applications to cscK metrics}
Finally we point out that our Aubin--Yau deformation viewpoint also has new applications for the problem of finding constant scalar curvature K\"ahler (cscK) metrics. Indeed, in \cite{DZ25} we introduced an increasing family of $K^\beta$ energies and showed that the existence of a unique cscK metric in $c_1(L)$ is equivalent to the properness of $K^\beta$ for some $\beta>0$. Combining our techniques with that of  \cite{DZ25} we obtain the following addendum to the YTD correspondence of \cite{DZ25} for polarized K\"ahler manifolds without non-trivial holomorphic vectorfields:

\begin{theorem}
\label{thm:cscK}
    Let $(X,L)$ be a polarized manifold. The following are equivalent.
    \begin{enumerate}
        \item There exists a unique cscK metric in $c_1(L)$.
        \item For some $\beta>0$ the $K^\beta$ energy is proper.
 \item $(X,L)$ is uniformly $K^\beta$-stable for some $\beta>0.$
        \item For some $\beta>0$ the $K^\beta$ energy is bounded from below.
        \item $(X,L)$ is $K^\beta$-semistable for some $\beta>0.$
    \end{enumerate}
\end{theorem}
We refer the reader to \S\ref{sec: cscK} for the definition of the $K^\beta$ energies, and to Definition \ref{def: K_beta_st} for the notion of $K^\beta$-(semi)stability. The equivalences between 1), 2), and 3) were established in \cite{DZ25}, the new equivalences are with 4) and 5). 

In a nutshell, the proof of Theorem \ref{thm:cscK} reduces to establishing the implication $(5)\Rightarrow(3)$, since the remaining equivalences follow from the results of \cite{DZ25}. For this, we will observe that if the slope of the $K^\beta$ energy along a non-trivial weak geodesic ray $\{u_t \}$ is non-negative, then the slope of the $K^{\beta'}$ energy along the same ray is strictly positive for every $\beta'>\beta$. This will follow from strict concavity of $H_u(\varepsilon)$ (recall \eqref{eq: H_u_def}).

Thus the main interest of the above result lies in the observation that, for sufficiently large \(\beta\), \(K^\beta\)-semistability is already equivalent to the existence of cscK metrics. This phenomenon appears to have no direct analogue for the more traditional notions of \(K\)-stability studied in \cite{Tian97,Don02,Sz06, StSz,Paul12, Der14,BHJ17,Li20-cscK, BJ25preprint, MPWN,T26}.

\paragraph{Acknowledgments.} The second author is grateful to J. Han and M. Xia for helpful discussions related to this work.  The first author was partially supported by NSF grant DMS-2405274. 
The second author is supported by Scientific Research Innovation Capability Support Project for Young Faculty SRICSPYF-ZY2025169 and NSFC grants 12571060, 12271038 and 12271040. 

After the completion of this work, we learned of the recently appeared manuscript \cite{LZ26} that also establishes the equivalences of Theorem \ref{thm:cscK}, and develops a version compatible with continuous automorphism groups. We also note another related recent work \cite{Pap26}, comparing K-stability and $K^\beta$-stability using algebraic methods.

\paragraph{AI use.} All arguments presented in this manuscript were developed and written up by the authors prior to April 2026. The second author gave talks about this work on several occasions, including Zhuhai's workshop on Geometric Analysis in May 2026. 
AI assistance was only used for improving presentation of the final manuscript.

\section{Preliminaries}
\label{sec: pre}

\subsection{Pluripotential theory}

Let $(X,\omega)$ be a compact K\"ahler manifold of dimension $n$.
The space of smooth K\"ahler potentials is defined as follows:
$$\mathcal H := \{u \in C^\infty(X,\RR) \textup{ s.t. } \omega_u :=\omega + \ddc u >0 \},$$
where $\ddc:=\frac{i}{2\pi}\partial\bar \partial.$ The metric completion of $\mathcal H$ with respect to the $d_1$ metric \cite{Dar15} is denoted by $\cE^1$, which is called the space of finite energy potentials \cite{GZ07}. Given $u,v\in\cE^1$, the $d_1$-distance between $u$ and $v$ is given as
\[
d_1(u,v)=I(u)+I(v)-2I(P(u,v)),
\]
where $I(\cdot)$ denotes the Monge--Amp\`ere energy (recall \eqref{eq: def_E}) and $P(u,v) \in \mathcal E^1$ is the largest potential in $\cE^1$ lying below $u$ and $v$.

Given $a,b \in [-\infty,\infty]$ with $a < b$, a path $(a,b)\ni t\mapsto u_t
\in \cE^1$ is called a finite energy subgeodesic segment if $U_z(x):=u_{-\log |z|^2}(x)$ (with $x\in X$ and $z\in\CC^*$) satisfies
 $
 p_1^*\omega+\mathrm{dd}^c_{X\times\CC}U_z\geq 0
 $
 as a current on $X\times\{e^{-b/2}<|z|<e^{-a/2}\}$, where $p_1:X\times\CC\to X$ is the projection map.

Let $h_0,h_1 \in \mathcal E^1$. As shown in \cite{Dar15}, a distinguished $d_1$-geodesic $(0,1) \ni t \mapsto h_t 
\in \mathcal 
E^1$ connecting $h_0,h_1$ can be obtained as:
\begin{equation}
\label{fegeod}
h_t := \textup{usc}\Big(\sup \{v_t \ | (0,1) \ni l \ \mapsto v_l  \in \mathcal E^1 \textup{ subgeodesic with } \limsup_{l \to 0,1} v_l \leq h_{0,1}\}\Big), \ t \in (0,1).
\end{equation}
We call \eqref{fegeod}
the  \emph{finite energy geodesic} connecting $h_0,h_1$.

In the case $a=0$ and $b=\infty$, the curve $(0,\infty) \ni t \mapsto u_t \in \cE^1$ is called a finite energy (sub)geodesic ray, which we will denote by $\{u_t\}$. All rays in this paper will be assumed to emanate from $u_0=0\in\cH$, meaning that $d_1(u_t,0) \to 0$ as $t\searrow 0$. It is well known that $t\mapsto I(u_t)$ is convex (resp. linear) when $\{u_t\}$ is a subgeodesic (resp. geodesic). A subgeodesic ray $\{u_t\}$ is called sublinear if $u_t\leq Ct+C$ for some $C>0$. 
Note that a geodesic ray is always sublinear.

Finally, given two geodesic rays $\{u_t\}$ and $\{v_t\}$, the chordal metric between them is defined as (see also \cite{DL19})
$$
d_1^c(\{u_t\},\{v_t\}):=\lim_{t\to\infty}\frac{d_1(u_t,v_t)}{t}.
$$
We denote by $\cR^1$ the space of all geodesic rays and equip it with the $d_1^c$-metric, which is a complete metric space.

\subsection{Test configurations and models}

Let $(X,L)$ be a complex projective manifold, with $L$ being an ample line bundle on $X$.

As in \cite[\S 2]{DZ25}, a \emph{test configuration} $(\cX,\cL)$ of $(X, L)$ is a $\CC^*$-equivariant partial compactification over $\CC$ of $(X, L)\times\CC^*$ which consists of a normal variety $\cX$, a flat projective morphism $\pi: \cX \to\CC$, a relatively semi-ample $\QQ$-line bundle $\cL\to \cX$, a $\CC^*$-action on $(\cX,\cL)$ lifting the standard one on $\CC$, and an identification of the fiber $(\mathcal X_1, \mathcal L_1)$ with $(X, L)$.

In practice, it is useful to compactify  a test configuration $(\cX,\cL)$ trivially over $\infty\in\PP^1$ (see \cite[Definition 2.7]{Xubook}), which we still denote as $(\cX,\cL)$. In this paper we assume that all test configurations $(\cX,\cL)$ are compactified in this sense, making intersection theory more convenient on $\cX$. 

Given two test configurations $(\cX,\cL)$ and $(\cX',\cL')$, we say $(\cX',\cL')$ dominates $(\cX,\cL)$ if there is a $\CC^*$-equivariant birational morphism $\cX'\xrightarrow{\pi}\cX$. So any test configuration is dominated by a smooth one.

In this work we will also work occasionally with \emph{models}, introduced in \cite[Definition 2.7]{Li20-cscK}. These are generalizations of test configurations $(\cX,\cL)$, where $\cL$ is only relatively big on $\cX$ (namely, $\cL+c\cX_0$ is big for some $c>0$). In \cite{DZ25} a special class of models was introduced, called \emph{log discrepancy models}. More precisely, given a smooth test configuration $(\cX,\cL)$ dominating $X\times\PP^1$ (via a map $\pi$), then a log discrepancy model of $(\cX,\cL)$ is of the form $(\cX,\cL+\varepsilon K^{\log}_{\cX/X\times\PP^1})$ for some $\varepsilon\in\QQ_{\geq 0}$, where $K^{\log}_{\cX/X\times\PP^1}:=K_\cX-\pi^*K_{X\times\PP^1}+\cX_{0,red}-\cX_0$.

\subsection{Revisiting the radial analysis of \cite[\S 6]{DZ25}}

Recall the definition of the  $\mathrm{Ent}^\beta$ functional and its radial version from \cite{DZ25}:
\begin{equation}\label{eq: beta_entropy_def}
\mathrm{Ent}^\beta(v) = \beta(I(v^\beta) - I(v)), \ \ \ \ \mathrm{Ent}^\beta\{u_t\}=\lim_{t\to\infty}\frac{\mathrm{Ent}^\beta(u_t)}{t},\qquad v \in \mathcal E^1, \ \{u_t\}\in\mathcal R^1.
\end{equation}
That the above limit exists follows from analysis in \cite[\S6]{DZ25}. We also proved the following in \cite[Theorem 6.1]{DZ25} (more details on these points will appear in the revision of \cite{DZ25}).
\begin{theorem}\label{thm: rad_Ent_beta=sup}
    For any $\beta>0$ and $\{u_t\} \in \mathcal R^1$,
    we have that
\begin{equation}
    \label{eq: rad_Ent_beta_formula}
    \mathrm{Ent}^\beta\{u_t\}=\sup_{\{v_t\}} F^\beta(\{v_t\},\{u_t\}):=\sup_{\{v_t\}}\left(L^\beta(\{v_t\},\{u_t\})+\beta (I\{v_t\}-I\{u_t\})\right),
\end{equation}
    where the sup is over all geodesic rays $\{v_t\}$ arising from test-configurations of $(X,L)$ (or more generally $\{v_t\} \in \mathcal R^1$). Here
    \begin{equation}
    \label{eq:def_L_beta}
    L^\beta(\{v_t\},\{u_t\}):=\liminf_{t\to\infty}\left(-\frac{1}{t}\log\int_Xe^{\beta(v_t-u_t)}\frac{\omega^n}{V}\right)
\end{equation}
for a sublinear subgeodesic ray $\{v_t\}$.
\end{theorem}

Recall from \cite[\S 6]{DZ25} the following related functional, for $\{u_t\} \in \mathcal R^1$ and $\{v_t\}$ sublinear finite energy subgeodesics :
$$\tilde L^\beta(\{v_t\},\{u_t\}): = \sup\left\{\tau\in\RR\ :\int_0^\infty e^{\tau t}\int_Xe^{\beta(v_t-u_t)}\omega^n dt<\infty\right\}.$$

We prove a series of lemmas about this latter functional before showing that it is equal to $L^\beta(\{v_t\},\{u_t\})$. For the sublinear finite energy rays considered below, $\widetilde L^\beta(\{v_t\},\{u_t\})$ is finite: the two required exponential bounds follow from sublinearity, Jensen's inequality, and \cite[Lemma 6.5]{DZ25}.

\begin{lemma}\label{lem: Dem_approx_exponent} Let $\{u_t\},\{v_t\}$ be sublinear finite energy subgeodesics with $\lim_{t \to \infty} \frac{\sup_X v_t}{t} =0$.
Then there exists $C := C(\{u_t\},\beta,\tilde L^\beta(\{v_t\},\{u_t\}) )>0$ such that
$$\tilde L^\beta(\{v_t\},\{u_t\}) \geq \tilde L^\beta(\{v^m_t\},\{u_t\}) \geq \tilde L^\beta(\{v_t\},\{u_t\}) - \frac{C}{m},$$
 for the Demailly approximations of $\{v^m_t\}$ of $\{v_t\}$. 
\end{lemma}

By the Demailly approximations $\{v^m_t\}$ of $\{v_t\}$ we mean the geodesic rays whose singularity is governed by $\frac{1}{m + m_0}\log \big|\mathcal I(mV_z)\big|$, as described in \cite[Section 2]{DZ25}. The constant $m_0 > \max\{\beta,1\}$ will be chosen in the argument below, it will depend on $\beta$, but it will not depend on $\{u_t\},\{v_t\}$.

\begin{proof} Let $\tilde L:= \tilde L^\beta(\{v_t\},\{u_t\})$. Since $0 \geq v^m_t \geq \frac{m}{m+ m_0}v_t \geq v_t$, we get that $\tilde L \geq \tilde L^\beta(\{v^m_t\},\{u_t\})$. Now we prove the other inequality.
Let $\varepsilon \in (0,1)$. We have that
$$\int_0^\infty e^{ (\tilde L-\varepsilon)t}\int_Xe^{\beta(v_t-u_t)}\omega^n dt<\infty$$

    Moreover, for some sufficiently large $m_0>\max\{\beta,1\}$ we have that
    $$
    \int_0^\infty e^{-t}\int_X e^{(m+m_0)v^m_t-m v_t}\omega^n dt<\infty, \ \forall m.
    $$
    According to \cite[Lemma 6.5]{DZ25}, for any $\alpha>0$ there exists $c_{\alpha,u}>0$ such that 
    $$
    \int_0^\infty e^{-c_{\alpha,u} t}\int_X e^{-\alpha u_t}\omega^n dt<\infty.
    $$
We fix $\alpha: = \frac{\beta m_0}{m_0 - \beta}$, and for  simplicity we introduce $Y:=X\times(0,\infty)$ and $d\mu:=\omega^n\wedge dt$. 
    
We choose $p,q,r>0$ so that $\frac{1}{p}+\frac{1}{q}+\frac{1}{r}=1$. Specifically,
$$
p:=\frac{m+m_0}{\beta},\ q:=\frac{m+m_0}{m},\ r:=\frac{m+m_0}{m_0-\beta}.
$$    
For this triplet we apply the H\"older inequality in the following way:
    \begin{equation*}
        \begin{split}
            \int_Y & e^{\beta(v^m_{t}-u_t) + \big((1-\frac{m_0}{m + m_0})(\tilde L - \varepsilon)-\frac{\beta m_0}{m + m_0}-c_{\alpha,u} \frac{m_0 - \beta}{m + m_0}\big) t}d\mu\\
            &= \int_Y e^{\beta(v^m_t-(1-\frac{m_0}{m + m_0})v_t)-\frac{\beta m_0}{m + m_0} t}\cdot e^{\beta(1-\frac{m_0}{m + m_0})(v_t-u_t)+(1-\frac{m_0}{m + m_0})(\tilde L - \varepsilon) t}\cdot e^{-\beta\frac{m_0}{m + m_0} u_t -c_{\alpha,u} \frac{m_0 - \beta}{m + m_0}t}d\mu\\
            &\leq\left(\int_Y e^{p\beta(v^m_t-(1-\frac{m_0}{m + m_0})v_t)-\frac{p\beta m_0}{m + m_0} t} d\mu\right)^{1/p}\left(\int_Y e^{ q\beta(1-\frac{m_0}{m + m_0})(v_t-u_t)+q(1-\frac{m_0}{m + m_0})(\tilde L - \varepsilon) t} d\mu\right)^{1/q}\\
            &\cdot \left(\int_Ye^{-r\beta\frac{m_0}{m + m_0} u_t -rc_{\alpha,u} \frac{m_0 - \beta}{m + m_0}t} d\mu \right)^{1/r}.\\
        \end{split}
    \end{equation*}

Now we point out that each of the integrals in the last product are finite:

\begin{flalign*}
\int_Y e^{p\beta(v^m_t-(1-\frac{m_0}{m + m_0})v_t)-\frac{p\beta m_0}{m + m_0} t} d\mu = \int_Y e^{(m+m_0)v^m_t-m v_t-m_0t} d\mu\leq \int_Y e^{(m+m_0)v^m_t-m v_t-t}d\mu< +\infty.
\end{flalign*}
\begin{flalign*}
\int_Y e^{ q\beta(1-\frac{m_0}{m + m_0})(v_t-u_t)+q(1-\frac{m_0}{m + m_0})(\tilde L - \varepsilon) t} d\mu =\int_Y e^{ \beta(v_t-u_t)+(\tilde L - \varepsilon) t} d\mu < + \infty.
\end{flalign*}
\begin{flalign*}
\int_Ye^{-r\beta\frac{m_0}{m + m_0} u_t -rc_{\alpha,u} \frac{m_0 - \beta}{m + m_0}t} d\mu  &= \int_Ye^{-\alpha u_t -c_{\alpha,u} t} d\mu < +\infty.
\end{flalign*}
Since $(1-\frac{m_0}{m + m_0})(\tilde L - \varepsilon)-\frac{\beta m_0}{m + m_0}-c_{\alpha,u} \frac{m_0 - \beta}{m + m_0} \geq \tilde L - \varepsilon - \frac{C}{m}$ for some constant $C:=C(\beta, u, \tilde L)$, and $\varepsilon >0$ is arbitrary, the result follows.
\end{proof}

\begin{remark} In the above result the condition $c := \lim_{t \to \infty} \frac{\sup_X v_t}{t} =0$ can be removed immediately, if we define the Demailly approximation $\{v^m_t\}$ of a general sublinear finite energy subgeodesic $\{v_t\}$ by 
$$v^m_t : = \tilde v^m_t + c t,$$
where $\tilde v_t := v_t - ct$. 
\end{remark}

The previous result immediately allows us to sharpen \cite[Proposition 6.2]{DZ25}.

\begin{corollary} \label{cor: liminf}
Let $\{u_t\}$ and $\{v_t\}$ be two sublinear finite energy subgeodesic rays and $\beta >0$. Then 
$$L^\beta(\{v_t\},\{u_t\})= \tilde L^\beta(\{v_t\},\{u_t\}).$$
In particular, it does not matter if one defines $F^\beta(\{\cdot\},\{u_t\})$ using $L^\beta(\{\cdot\},\{u_t\})$ or $\tilde L^\beta(\{\cdot\},\{u_t\})$. 
\end{corollary}

\begin{proof} After normalization, we can assume that $\lim_t \frac{\sup_X v_t}{t} =0$.

For any $\tau < L^\beta(\{v_t\},\{u_t\})$, pick $\varepsilon>0$ such that $\tau+\varepsilon<L^\beta(\{v_t\},\{u_t\})$. Then for $t\gg 1$,
    $$
    -\log\int_X e^{\beta(v_t-u_t)}\frac{\omega^n}{V}\geq\tau t+\varepsilon t.
    $$
    Thus for $t_0$ big enough,
    $
    \int_{t_0}^\infty e^{\tau t}\int_X e^{\beta(v_t-u_t)}\frac{\omega^n}{V}\,dt\leq\int_{t_0}^\infty e^{-\varepsilon t}dt<\infty.
    $
    So, $L^\beta(\{v_t\},\{u_t\})\leq \tilde L^\beta(\{v_t\},\{u_t\}).$

For the converse inequality, we consider a Demailly approximation $\{v^m_t\}$ of $\{v_t\}$ as in the previous result. Since $v^m_t \geq v_t$, using the previous lemma, we can start the following chain of inequalities:
$$\tilde L^\beta(\{v_t\},\{u_t\}) \leq \tilde L^\beta (\{v^m_t\},\{u_t\}) + \frac{C}{m}= L^\beta (\{v^m_t\},\{u_t\}) + \frac{C}{m} \leq L^\beta (\{v_t\},\{u_t\}) + \frac{C}{m},$$
where we also used \cite[Proposition 6.2]{DZ25} in the middle identity, since $\{v_t^m\}$ are bounded geodesics.
Letting $m \to \infty$, the result follows.
\end{proof}

\subsection{Non-Archimedean geometry}
\label{sub: NA_pre}
Let $(X,L)$ be a complex projective manifold of dimension $n$. In \S \ref{subset: NA} we already introduced several notions from NA geometry. Here let us recap a few more, for later discussions.

Fix a K\"ahler form $\omega\in c_1(L)$. Then any geodesic ray $\{u_t\}\subset\cE^1$ gives rise to an element $U\in\cE^{1,\mathrm{NA}}$ (see \cite[Theorem 6.4]{BBJ18}), which is essentially determined by the Lelong numbers of $\{u_{-\log|z|^2}\}$ near $z=0$ (see \cite[Definition 4.2]{BBJ18}). Conversely, any $U\in \cE^{1,\mathrm{NA}}$ uniquely determines a maximal geodesic ray $\{u_t\}\subset\cE^1$, in the sense of \cite[\S 6.4]{BBJ18}.

 Let $(\cX,\cL)$ be a test configuration of $(X,L)$. Then one can attach a geodesic ray $\{u_t\}\subset\cE^1$ to it \cite{PS07a,PS10, RWN14,Ber16,CTW18}, whose associated NA potential $U$ belongs to $\cH^{\mathrm{NA}}$. Conversely, any $U\in\cH^{\mathrm{NA}}$ arises in this way.

Given a test configuration $(\cX,\cL)$, with $U\in\cH^{\mathrm{NA}}$ being its associated potential, then the NA Monge--Amp\`ere energy of $U$ is given by
$$
I^{\mathrm{NA}}(U):=\frac{\cL^{n+1}}{(n+1)V},
$$
where $V=L^n$. This further extends to any $U\in \mathrm{PSH}^{\mathrm{NA}}$ by putting
$$
I^{\mathrm{NA}}(U):=\inf\{I^{\mathrm{NA}}(W):W\in\cH^{\mathrm{NA}},U\leq W\}.
$$
Then the space $\cE^{1,\mathrm{NA}}$ of NA potentials with finite energies can be described as
$$
\cE^{1,\mathrm{NA}}:=\{U\in\mathrm{PSH}^{\mathrm{NA}}: I^{\mathrm{NA}}(U)>-\infty\}.
$$

Now let $(\cX,\cL)$ be a smooth test configuration with simple normal crossing central fiber $\cX_0$. Then $\cX_0$ determines a dual complex $\Delta_\cX$, whose points induce quasi-monomial valuations on $X$. In this way, $\Delta_\cX$ is viewed as a subset of $X^{an}$. Moreover, there is a retraction map $r_\cX:X^{an}\to \Delta_\cX$, making $A_X\circ r_\cX$ a continuous function on $X^{an}$. For more details, we refer to \cite[\S 4]{BoJ18a} and \cite[Appendix A]{BoJ18}. 

Assume that $U\in\cH^{\mathrm{NA}}$ corresponds to some smooth test configuration $(\cX,\cL)$. Then its NA Monge--Amp\`ere measure $\mathrm{MA}^{\mathrm{NA}}(U)$ is a Dirac type measure supported on the vertices $v_E$ of $\Delta_\cX$, where each $v_E$ corresponds to an irreducible component $E\subset\cX_0$. If we write $\cX_0=\sum_E b_E E$, then (see \cite[\S 6.7]{BHJ17})
$$
\mathrm{MA}^{\mathrm{NA}}(U):=\sum_E V^{-1}b_E(\cL^n\cdot E)\delta_{v_E}.
$$

For a general $U\in\cE^{1,\mathrm{NA}}$, one can pick a decreasing net $\cH^{\mathrm{NA}}\ni U^k\to U$. Then the NA Monge--Amp\`ere measure $\mathrm{MA}^{\mathrm{NA}}(U)$ is the weak limit of $\mathrm{MA}^{\mathrm{NA}}(U^k)$ (see \cite[Corollary 6.12]{BoJ18a}).

Finally, recall that a model of $(X,L)$ also induces an element in $\cE^{1,\mathrm{NA}}$. For us, we will only need the case of log discrepancy models: given a smooth test configuration $(\cX,\cL)$ dominating $X\times\PP^1$, with $U\in\cH^{\mathrm{NA}}$ being its associated NA potential, then the NA potential of the log discrepancy model $(\cX,\cL+\varepsilon K^{\log}_{\cX/X\times\PP^1})$ is given by the following envelope (see \cite[Remark 9.2]{DZ25}):
$$
P(U+\varepsilon A_X\circ r_\cX):=\sup\{V:V\in \mathrm{PSH}^{\mathrm{NA}}, V\leq U+\varepsilon A_X\circ r_\cX\text{ on }X^{an}\}.
$$
Moreover, it follows from \cite[(1)]{Li21} that the NA Monge--Amp\`ere measure of this envelope is supported on the dual complex $\Delta_\cX$.

\section{K\"ahler entropy}

This part is devoted to proving Theorem \ref{thm:main_f'}. Let $(X,\omega)$ be a compact K\"ahler manifold of dimension $n$.
For any $u\in\cE^{1}$ and $\varepsilon>0$, let $u_\varepsilon$ be the solution of \eqref{eq: def_u_e}.
Let us recap several necessary estimates for $u_\varepsilon$ from \cite[\S 3]{DZ25}.

\begin{lemma}
	\label{lem: mono_conv}
	If $u_i\searrow u$ in $\cE^1$, then $u_{i,\varepsilon}\searrow u_\varepsilon$ for any $\varepsilon>0$.
\end{lemma}

\begin{lemma}
	\label{lem: u_e_d1_conv_u}
	One has $u_\varepsilon\rightarrow_{d_1}u$ as $\varepsilon\to 0$.
\end{lemma}

\begin{lemma}
	\label{lem: Ent_ue_conv_Ent_u}
	One has $\mathrm{Ent}(u_\varepsilon)\to \mathrm{Ent}(u)$ as $\varepsilon\to 0$.

\end{lemma}

\begin{lemma}
	\label{lem: Ent>Ent_e>Ent_e}
	One has $\mathrm{Ent}(u)\geq \frac{1}{\varepsilon}(I(u_\varepsilon )-I(u))\geq \mathrm{Ent}(u_\varepsilon)$.
\end{lemma}

\begin{lemma}
	\label{lem:lim_Ent_e=Ent}
	One has
	$\mathrm{Ent}(u)=\lim_{\varepsilon\to 0}\frac{1}{\varepsilon}(I(u_\varepsilon )-I(u))$.
\end{lemma}

We now fix some $u\in\cE^1$ and investigate the function:
\begin{equation*}
f(\varepsilon):=
    \begin{cases}
        I(u_\varepsilon),\ \varepsilon>0,\\
        I(u),\ \varepsilon=0.
        
    \end{cases}
\end{equation*}
We aim to show the following.

\begin{theorem}[=Theorem \ref{thm:main_f'}]
\label{thm: f'}
    The function $f(\varepsilon)$ is concave, non-decreasing and belongs to  $C^1([0,\infty))$. Moreover,
    $$
    f'(\varepsilon)=\mathrm{Ent}(u_\varepsilon)\text{ for any $\varepsilon>0$ and }
    f'_+(0)=\mathrm{Ent}(u).
    $$
\end{theorem}

Here $f\in C^1([0,\infty))$ means that $f\in C^1(0,\infty )$ and $\lim_{\varepsilon\to 0}f'(\varepsilon )=f'_+(0)$. Note that $f'_+(0)$ denotes the right derivative of $f$ at $\varepsilon= 0$ that is allowed to be $+\infty$ (this happens precisely when $\mathrm{Ent}(u)=+\infty$). Indeed, by Lemma \ref{lem: Ent>Ent_e>Ent_e} and Lemma \ref{lem:lim_Ent_e=Ent},
$$
f'_+(0)=\lim_{\varepsilon\to 0^+}\frac{I(u_\varepsilon)-I(u)}{\varepsilon}=\mathrm{Ent}(u).
$$

An immediate consequence of Theorem \ref{thm: f'} is the following result.
\begin{corollary}
    For any $u\in\cE^1$, $\mathrm{Ent}(u_\varepsilon)\nearrow \mathrm{Ent}(u)$ as $\varepsilon\searrow 0$.
\end{corollary}

The nontrivial part of Theorem \ref{thm: f'} is to show that $f$ is $C^1$ for $\varepsilon\in (0,\infty)$ and
$$
f'(\varepsilon)=\mathrm{Ent}(u_\varepsilon).
$$
The proof takes several steps. Let us first pick a regularization sequence $\cH\ni u_i\searrow u$. We define as above a sequence of functions
\begin{equation*}
f_i(\varepsilon):=
    \begin{cases}
        I(u_{i,\varepsilon}),\ \varepsilon>0,\\
        I(u_i),\ \varepsilon=0
        
    \end{cases}
\end{equation*}
We first argue that Theorem \ref{thm: f'} holds for $f_i(\varepsilon)$ and then let $i\to\infty$. 

Note that $u_i\in C^\infty$, $\{u_{i,\varepsilon}\}_{\varepsilon\in\RR_{>0}}$ is a smooth family of solutions to \eqref{eq: def_u_e},
so $f_i$ is smooth for $\varepsilon\in(0,\infty)$. Its continuity at $\varepsilon=0$ follows from Lemma \ref{lem: u_e_d1_conv_u}.
 Moreover, by differentiating $\varepsilon$ in \eqref{eq: def_u_e} we obtain that
 \begin{equation}
     \label{eq:Lapdot_u_e}
     \Delta_{i,\varepsilon}\dot u_{i,\varepsilon}=-\frac{1}{\varepsilon^2}(u_{i,\varepsilon}-u_i)+\frac{1}{\varepsilon}\dot u_{i,\varepsilon},\ \varepsilon>0,
 \end{equation}
where $\Delta_{i,\varepsilon}$ denotes the Laplacian of $\omega_{u_{i,\varepsilon}}$ and $\dot u_{i,\varepsilon}:=\frac{\partial}{\partial\varepsilon}u_{i,\varepsilon}$. Integrating the above equation against $\omega^n_{u_{i,\varepsilon}}$ implies that
\begin{equation}
    \label{eq: fi'=ent}
    f_i'(\varepsilon)=\frac{1}{V}\int_X\dot u_{i,\varepsilon}\omega^n_{u_{i,\varepsilon}}=\frac{1}{V}\int_X\frac{1}{\varepsilon}(u_{i,\varepsilon}-u_i)\omega^n_{u_{i,\varepsilon}}=\mathrm{Ent}(u_{i,\varepsilon}),\ \varepsilon>0.
\end{equation}
At the endpoint $\varepsilon=0$, Lemma \ref{lem:lim_Ent_e=Ent} implies that
$$
f'_{i,+}(0)=\lim_{\varepsilon\to 0}\frac{1}{\varepsilon}(I(u_{i,\varepsilon})-I(u_i))=\mathrm{Ent}(u_i).
$$
Combining this with Lemma \ref{lem: Ent_ue_conv_Ent_u}, we deduce that
$$
f'_{i,+}(0)=\mathrm{Ent}(u_i)=\lim_{\varepsilon\to0}\mathrm{Ent}(u_{i,\varepsilon})=\lim_{\varepsilon\to0} f'_i(\varepsilon ).
$$
The above analysis yields the following.

\begin{lemma}
	One has $f_i\in C^1([0,\infty))$.

\end{lemma}

The next result follows from a direct calculation.

\begin{lemma}
	The function $f_i$ is concave on $[0,\infty )$.
\end{lemma}

\begin{proof}
	It suffices to show that $f''_i(\varepsilon )\leq 0$ for any $\varepsilon>0$.
	For this, we will use \eqref{eq:Lapdot_u_e} and \eqref{eq: fi'=ent} to write
	\begin{equation*}
		\begin{split}
			V\cdot f''_i(\varepsilon )&=-\int_X\frac{1}{\varepsilon^2}(u_{i,\varepsilon}-u_i)\omega^n_{u_{i,\varepsilon}}+\int_X\frac{1}{\varepsilon}\dot u_{i,\varepsilon}\omega^n_{u_{i,\varepsilon}}+\int_X\frac{1}{\varepsilon}(u_{i,\varepsilon}-u_i)\Delta_{i,\varepsilon}\dot u_{i,\varepsilon}\omega^n_{u_{i,\varepsilon}}\\
			&=\int_X \Delta_{i,\varepsilon} \dot u_{i,\varepsilon}\omega^n_{u_{i,\varepsilon}}+\int_X(\dot u_{i,\varepsilon}-\varepsilon \Delta_{i,\varepsilon} \dot u_{i,\varepsilon} ) \Delta_{i,\varepsilon} \dot u_{i,\varepsilon}\omega^n_{u_{i,\varepsilon}}\\
			&=-\int_X|\nabla_{i,\varepsilon} \dot u_{i,\varepsilon} |^2\omega^n_{u_{i,\varepsilon}}-\varepsilon\int_X(\Delta_{i,\varepsilon} \dot u_{i,\varepsilon} )^2\omega^n_{u_{i,\varepsilon}}\leq 0.
		\end{split}
	\end{equation*}
	This completes the proof.
\end{proof}

Therefore, we conclude that Theorem \ref{thm: f'} holds for the sequence $\{f_i\}$.
By Lemma \ref{lem: mono_conv}, we know that $f_i\to f$ pointwise on $[0,\infty)$, so we deduce the following.

\begin{corollary}
    The function $f$ is non-decreasing and concave on $[0,\infty)$. 
\end{corollary}

Next we move on to show that $f'_i\to f'$ on $(0,\infty)$.

\begin{lemma}
    For any $\varepsilon>0$, one has $\lim_{i\to\infty}\mathrm{Ent}(u_{i,\varepsilon})=\mathrm{Ent}(u_\varepsilon).$
\end{lemma}

\begin{proof}
    Note that
    $$
    \mathrm{Ent}(u_{i,\varepsilon})=\frac{1}{\varepsilon V}\int_X(u_{i,\varepsilon}-u_i)e^{\frac{1}{\varepsilon}(u_{i,\varepsilon}-u_i)}\omega^n = \frac{1}{\varepsilon V}\int_X(u_{i,\varepsilon}-u_i)\omega_{u_{i,\varepsilon}}^n
    $$
    and
    $$
    \mathrm{Ent}(u_{\varepsilon})=\frac{1}{\varepsilon V}\int_X(u_{\varepsilon}-u)e^{\frac{1}{\varepsilon}(u_{\varepsilon}-u)}\omega^n=\frac{1}{\varepsilon V}\int_X(u_{\varepsilon}-u)\omega_{u_\varepsilon}^n.
    $$
    The assertion then follows from dominated convergence. Indeed, we can assume that $\sup u_i\leq \sup u+100$ for all $i\geq 1$, so that
    $$
    u_\varepsilon-\sup u-100\leq u_{i,\varepsilon}-u_i\leq \sup u+100-u.
    $$
    Moreover, by Ko\l odziej's $L^\infty$ estimate \cite{Kol98}, $u_\varepsilon>-C_\varepsilon$ for some $C_\varepsilon>0$. Thus
    $$
    |u_{i,\varepsilon}-u_i|\leq |u|+C_\varepsilon'.
    $$
    So we have that
    $$
    |u_{i,\varepsilon}-u_i|e^{\frac{1}{\varepsilon}(u_{i,\varepsilon}-u_i)}\leq (|u|+C'_\varepsilon)e^{\frac{1}{\varepsilon}(\sup u+100-u)}.
    $$
    We claim that the right hand side is integrable against $\omega^n$. Indeed, by \cite{Ze01} $e^{-pu}\in L^1$ for all $p>1$. So by \cite{Kol98} for some suitable $A_\varepsilon\in\RR$ one can find some $\psi\in\cE^1$ such that $\omega^n_\psi=e^{A_\varepsilon-\frac{1}{\varepsilon}u}\omega^n$. Then the claim follows from \cite[Corollary 3.50]{DarvasSurvey}. Thus we conclude from dominated convergence, since $(u_{i,\varepsilon}-u_i)e^{\frac{1}{\varepsilon}(u_{i,\varepsilon}-u_i)}$ converges to $(u_{\varepsilon}-u)e^{\frac{1}{\varepsilon}(u_{\varepsilon}-u)}$ pointwise. 
\end{proof}

\begin{lemma}
    The map $(0,\infty)\ni\varepsilon\mapsto \mathrm{Ent}(u_\varepsilon)$ is continuous.
\end{lemma}

\begin{proof}
Fix $\varepsilon>0$.
    By the stability estimate \cite[Remark 3.2]{GLZ18} (cf. also \cite[Lemma 3.5]{DZ25}), it is clear that $u_{\varepsilon+s}\rightarrow_{C^0}u_\varepsilon$ as $s\to 0$.
    Writing
    $$
    \mathrm{Ent}(u_{\varepsilon+s})=\frac{1}{(\varepsilon+s)V}\int_X(u_{\varepsilon+s}-u)e^{\frac{1}{\varepsilon+s}(u_{\varepsilon+s}-u)}\omega^n
    $$
    and
    $$
    \mathrm{Ent}(u_{\varepsilon})=\frac{1}{\varepsilon V}\int_X(u_{\varepsilon}-u)e^{\frac{1}{\varepsilon}(u_{\varepsilon}-u)}\omega^n,
    $$
    the assertion follows from dominated convergence, as in the previous proof.
\end{proof}

\begin{proof}
    [Proof of Theorem \ref{thm: f'}]

    Now for any $\varepsilon>0$ and $s\in(0,\varepsilon)$, we can write
\begin{equation*}
    \begin{split}
        f'_+(\varepsilon)&\geq \frac{f(\varepsilon+s)-f(\varepsilon)}{s}=\lim_{i\to\infty}\frac{f_i(\varepsilon+s)-f_i(\varepsilon)}{s}\\
        &\geq \lim_{i\to\infty}f'_i(\varepsilon+s)=\lim_{i\to\infty} \mathrm{Ent}(u_{i,\varepsilon+s})=\mathrm{Ent}(u_{\varepsilon+s}).\\
    \end{split}
\end{equation*}
Similarly,
\begin{equation*}
    \begin{split}
        f'_-(\varepsilon)&\leq \frac{f(\varepsilon)-f(\varepsilon-s)}{s}=\lim_{i\to\infty}\frac{f_i(\varepsilon)-f_i(\varepsilon-s)}{s}\\
        &\leq \lim_{i\to\infty}f'_i(\varepsilon-s)=\lim_{i\to\infty} \mathrm{Ent}(u_{i,\varepsilon-s})=\mathrm{Ent}(u_{\varepsilon-s}).\\
    \end{split}
\end{equation*}

By the previous lemma we conclude that $f'(\varepsilon)=\mathrm{Ent}(u_\varepsilon)$ for all $\varepsilon>0.$ Combining with $f'_+(0)=\mathrm{Ent}(u)=\lim_{\varepsilon\to 0^+}\mathrm{Ent}(u_\varepsilon)$ (Lemma \ref{lem: Ent_ue_conv_Ent_u}), 
we finish the proof.
\end{proof}

Summing up, in the above discussion we have shown that
$$
f(\varepsilon)=\lim_{i\to\infty}f_i(\varepsilon)\text{ for all }\varepsilon\geq 0,
$$
and
$$
\mathrm{Ent}(u_\varepsilon)=f'(\varepsilon)=\lim_{i\to\infty}f'_i(\varepsilon)=\lim_{i\to\infty}\mathrm{Ent}(u_{i,\varepsilon}
)\text{ for all }\varepsilon>0.
$$
But at the endpoint $\varepsilon=0$ it may happen that
$$
f'_+(0)\neq \lim_{i\to\infty}f'_{i,+}(0).
$$
In other words, it is possible that
$$
\mathrm{Ent}(u)\neq \lim_{i\to\infty} \mathrm{Ent}(u_i).
$$
However, due to $\mathrm{Ent}(u)=\lim_{\varepsilon\to 0^+}\mathrm{Ent}(u_\varepsilon)$, and $\mathrm{Ent}(u_\varepsilon)=\lim_{i\to\infty}\mathrm{Ent}(u_{i,\varepsilon})$, we reproduce the following regularization result contained in \cite[Theorem 3.2]{BDL17}.

\begin{corollary}
    One can find some $\varepsilon_i\searrow 0$ such that the sequence $v_i:=u_{i,\varepsilon_i}\in\cH$ satisfies
    $d_1(v_i,u)\to 0$ and
$
\mathrm{Ent}(u)=\lim_{i\to\infty}\mathrm{Ent}(v_i).
$
\end{corollary}

\section{Radial entropy}
\label{sec: rad}

This section is devoted to proving Theorem \ref{thm:intro_radial} and Theorem \ref{thm: intro_slope}.

\paragraph{The radial Aubin--Yau deformation.}
In what follows let $\{u_t\} \in \mathcal R^1$ with $\sup_X u_t=0$. We introduce the following envelope of finite energy subgeodesics for $t \in [0,\infty)$:
\begin{equation}\label{eq: P_beta_u_def}
P(u)_t^\beta:=\textup{usc}\left(\sup v_t\right),
\end{equation}
where the supremum is taken  over all sublinear subgeodesic rays $[0,\infty) \ni t \to v_t\in \mathcal E^1$ with $v_0 =0$ and $\tilde L^\beta(\{v_t\},\{u_t\}) \geq 0$.

A basic computation shows the following about rays that are candidates for $\{P(u)_t^\beta\}$:

\begin{lemma} \label{lem: max_lemma} If $\tilde L^\beta(\{v_t\},\{u_t\}) \geq 0,\tilde L^\beta(\{w_t\},\{u_t\}) \geq 0$ then $\tilde L^\beta(\{\max(v_t,w_t)\},\{u_t\}) \geq 0$.
\end{lemma}

\begin{proof} For $\tau < 0$ we have that
\begin{flalign*}
\int_0^\infty e^{\tau t}\int_Xe^{\beta(\max(v_t,w_t)-u_t)}\omega^n dt < \int_0^\infty e^{\tau t}\int_Xe^{\beta(v_t-u_t)}\omega^n dt + \int_0^\infty e^{\tau t}\int_Xe^{\beta(w_t-u_t)}\omega^n dt<\infty,
\end{flalign*}
finishing the proof.
\end{proof}
Later we show that the transfinite version of this lemma also holds. Before that we list some basic properties of $\{P(u)_t^\beta\}$:

\begin{lemma}\label{lem: sup_P} (i) $\sup_X P(u)_t^\beta = 0$ for all $t \geq 0$.\\
(ii) If $\beta \leq \beta'$ then $P(u)_t^\beta \geq P(u)_t^{\beta'}$.
\end{lemma}
\begin{proof} (i) Let us assume that $\{v_t\}_t$ is a candidate for $\{P(u)_t^\beta\}$, i.e., $\tilde L^\beta(\{v_t\},\{u_t\}) \geq 0$. We also assume that  $\sup_X v_t >0$ for some $t>0$. Since $t \to \sup_X v_t$ is convex, we get that $c:=  \lim_{t \to \infty} \frac{\sup_X v_t}{t} >0$.

Let $\varepsilon \in (0,c)$. We claim that 
\begin{equation}\label{eq: liminf_lower_bound}
\liminf_{t \to \infty} \int_X e^{\beta(v_t-\varepsilon t)}\frac{\omega^n}{V} \geq 1.
\end{equation}
Let us suppose otherwise. By Jensen's inequality there exists $t_j \to \infty$ such that
$$\frac{\beta}{V}\int_Xv_{t_j}\omega^n-\beta\varepsilon t_j \leq \log\int_Xe^{\beta(v_{t_j}-\varepsilon t_j)}\frac{\omega^n}{V} \leq0.$$
Since $\int_X v_{t_j} - \sup_X v_{t_j}$ is uniformly bounded, we obtain that $c =\lim_{j \to \infty} \frac{\sup_X v_{t_j}}{t_j} \leq \varepsilon$, a contradiction.

Suppose that $\tau \in (0,\varepsilon\beta )$. Then using \eqref{eq: liminf_lower_bound} we conclude that
\begin{flalign*}
\int_0^\infty e^{-\tau t}\int_Xe^{\beta(v_t-u_t)} \omega^n dt&=\int_0^\infty e^{\beta \varepsilon t - \tau t} \int_Xe^{\beta(v_t - \varepsilon t -u_t)} \omega^n dt\\
& \geq V\int_0^\infty e^{(\beta\varepsilon-\tau)t}\int_Xe^{\beta(v_t-\varepsilon t)}\frac{\omega^n}{V}dt=\infty,
\end{flalign*}
which contradicts $\tilde L^\beta(\{v_t\},\{u_t\})\geq0$. Consequently every candidate satisfies $\sup_Xv_t\leq0$. Since $\{u_t\}$ is itself a candidate and $\sup_Xu_t=0$, it follows that $\sup_XP(u)_t^\beta=0$.\smallskip 

\noindent (ii) For this we use the H\"older inequality and Corollary \ref{cor: liminf}:
\begin{flalign*}
L^\beta(\{v_t\},\{u_t\})&=\liminf_{t\to\infty}\left(-\frac{1}{t}\log\int_Xe^{\beta(v_t-u_t)}\frac{\omega^n}{V}\right) \geq \liminf_{t\to\infty}\left(-\frac{\beta}{t\beta'}\log\int_Xe^{{\beta'}(v_t-u_t)}\frac{\omega^n}{V}\right)\\
&= \frac{\beta}{\beta'}L^{\beta'}(\{v_t\},\{u_t\}).
\end{flalign*}
So if $\{v_t\}$ is a candidate for $\{P(u)^{\beta'}_t\}$ then it is candidate for $\{P(u)^{\beta}_t\}$ too.
\end{proof}

Next we prove the transfinite version of Lemma \ref{lem: max_lemma}:

\begin{lemma} \label{lem: increasing_lemma}Suppose that $\tilde L^\beta(\{v^j_t\},\{u_t\}) \geq 0$ with $\{v^j_t\}$ an increasing sequence of sublinear finite energy subgeodesics. Then for $v_t:= \textup{usc}(\lim_j v^j_t)$  we have that $\tilde L^\beta(\{v_t\},\{u_t\}) \geq 0$.
\end{lemma}

Note that it follows from the argument of the previous Lemma \ref{lem: sup_P} that $\sup_X v^j_t \leq 0$, hence $\{v_t\}$ is well defined in the above statement. 

\begin{proof} We have that $\tilde L^\beta(\{v^1_t\},\{u_t\}) \geq \tilde L^\beta(\{v^j_t\},\{u_t\}) \geq 0$, so the sequence of constants $\tilde L^\beta(\{v^j_t\},\{u_t\})$ is uniformly bounded.

Let $\{w^{m,j}_t\}$ and $\{w^m_t\}$ be the Demailly approximations of $\{v^j_t\}$ and $\{v_t\}$, respectively. From Lemma \ref{lem: Dem_approx_exponent} we obtain that 
\begin{flalign}\label{eq: int_ineq1}
\tilde L^\beta(\{w^{m,j}_t\},\{u_t\}) \geq \tilde L^\beta(\{v^j_t\},\{u_t\}) - \frac{C}{m} \geq - \frac{C}{m}.
\end{flalign}

Due to Lemma \ref{lem: Dem_approx_exponent} the constant $C$ can be chosen independently of $m,j$, since all the constants $\tilde L^\beta(\{v^j_t\},\{u_t\})$ are uniformly bounded.

Recall that the singularity type of the potentials $W^{m,j}_z = w^{m,j}_{-\log|z|^2}$ are governed by $\frac{1}{m + m_0}\log \big|\mathcal I(mV^j_z)\big|$. Since $V^j_z \nearrow V_z$, by the Guan--Zhou strong openness theorem \cite{GZ15}, we get that  $\mathcal I(m V^j_z) = \mathcal I(m V_z)$ for all $j \geq j(m)$. Hence, for any fixed $m$ we get that $W^{m,j}_z=W^m_z$ for $j \geq j(m)$. 

Putting this into \eqref{eq: int_ineq1}, we obtain that 
$$\tilde L^\beta(\{v_t\},\{u_t\}) \geq \tilde L^\beta(\{w^m_t\},\{u_t\})  = \tilde L^\beta(\{w^{m,j}_t\},\{u_t\})\geq - \frac{C}{m}.$$
Letting $m \to \infty$ concludes the argument.
\end{proof}

\begin{theorem}\label{thm: maximizer_main}  Let $\{u_t\} \in \mathcal R^1$ and $\beta>0$. Then \smallskip\\
(i) $\{P(u)^\beta_t\}$ is a bounded geodesic ray.\smallskip\\
(ii) $\{P(u)^\beta_t\}$ is the unique maximizing ray of $F^\beta(\{\cdot\},\{u_t\})$, satisfying $\tilde L^\beta(\{P(u)^\beta_t\},\{u_t\})=0$.\smallskip\\
(iii) If $\{w_t^m\}$ are Demailly approximations of $\{u_t\}$, then
$\displaystyle\sup_{t>0}t^{-1}\|w_t^m-P(u)_t^\beta\|_{L^\infty(X)}\to0$.
\end{theorem}
\begin{proof} 
First we argue that $\{P(u)_t^\beta\}$ is a geodesic ray. By Choquet's lemma and Lemma \ref{lem: max_lemma}, there exists $\{w^j_t\}$, an increasing sequence of sublinear finite energy subgeodesics, such that $\tilde L^\beta(\{w^j_t\}, \{u_t\}) \geq 0$ and $w^j_t \nearrow P(u)_t$. By Lemma \ref{lem: increasing_lemma} we obtain that $\tilde L^\beta(\{P(u)_t^\beta\}, \{u_t\}) \geq 0$, i.e., $\{P(u)_t^\beta\}$ is a candidate for itself. But then $\{P(u)_t^\beta\}$ must be a geodesic ray. Indeed, suppose it is not, then for $0<a<b$ let $\{g_t\}$ be the subgeodesic that agrees with $\{P(u)_t^\beta\}$ outside $[a,b]$ and is equal to the weak geodesic joining $P(u)_a^\beta$ and $P(u)_b^\beta$ on $[a,b]$. By the comparison principle, $g_t\geq P(u)^\beta_t$. Moreover, the two rays agree outside the compact interval $[a,b]$, so their exponential integrability thresholds coincide. Hence $\tilde L^\beta(\{g_t\},\{u_t\})=\tilde L^\beta(\{P(u)^\beta_t\},\{u_t\})\geq0$, and therefore $g_t=P(u)^\beta_t$.

Since $P(u)^\beta_t \geq u_t$ we obtain that $0= \tilde L^\beta(\{u_t\},\{u_t\}) \geq \tilde L^\beta(\{P(u)^\beta_t\},\{u_t\}) \geq 0.$

Suppose that $\{v_t\} \in \mathcal R^1$ is a maximizer of $F^\beta(\{\cdot\},\{u_t\})$ with $\tilde L^\beta(\{v_t\},\{u_t\})=0$. Such a maximizer is guaranteed to exist by results developed in \cite[Section 6]{DZ25} (this was not explicitly recorded there but will appear in its revision). 

Let $v^m_t$ be a Demailly approximation of $\{v_t\}$. By Lemma \ref{lem: Dem_approx_exponent}, we get that  
$$0 \geq c_m: = \tilde L^\beta(\{v^m_t\}, \{u_t\})  \to 0.$$
Since, $\tilde L^\beta(\{v^m_t + \frac{c_m}{\beta} t \}, \{u_t\})=0$, we get that $P(u)^\beta_t \geq v^m_t + \frac{c_m}{\beta} t$, in particular $\{P(u)^\beta_t\}$ is a bounded geodesic ray. Letting $m \to \infty$, it results that $P(u)^\beta_t\geq v_t$. But this implies that $I\{P(u)^\beta_t\} \geq I\{v_t\} $, giving that
$$F^\beta(\{P(u)^\beta_t\},\{u_t\}) \geq F^\beta_u\{v_t\}.$$
Since $v_t$ is a maximizer of $F^\beta(\{\cdot \},\{u_t\}) $, we get that $\{P(u)^\beta_t\}$ must be a maximizer as well, hence there is equality above, hence also $I\{P(u)^\beta_t\} = I\{v_t\}$, hence $P(u)^\beta_t = v_t$, as desired. 

Let $\{w^m_t\}$ be a Demailly approximation of $\{P(u)^\beta_t\}$. By the above argument we get that
$$w^m_t \geq P(u)^\beta_t \geq w^m_t + \frac{c_m}{\beta} t,$$
where $0\geq c_m\to0$. Hence
$\sup_{t>0}t^{-1}\|w_t^m-P(u)_t^\beta\|_{L^\infty(X)}\to0$.
\end{proof}

We note the following elementary estimate that will be needed later:

\begin{lemma}\label{lem: entropy_est} Let $\{u_t\}\in\mathcal R^1$ and $\beta>0$. Then

$$\mathrm{Ent}\{P(u)^\beta_t\} \geq \limsup_{t \to \infty} \frac{\beta}{tV}\int_X(P(u)^\beta_t-u_t)\omega^n_{P(u)^\beta_t}.$$
\end{lemma}
\begin{proof}Since $L^\beta(\{P(u)^\beta_t\},\{u_t\})=\tilde L^\beta(\{P(u)^\beta_t\},\{u_t\})=0$, for any $\varepsilon>0$ there exists $t_\varepsilon$ so that
$$\varepsilon\geq\frac{1}{t}\log\int_Xe^{\beta(P(u)^\beta_t-u_t)}\frac{\omega^n}{V},\qquad t\geq t_\varepsilon.$$
On the other hand, applying the Jensen inequality we obtain 
\begin{flalign*}
\log\int_Xe^{\beta(P(u)^\beta_t-u_t)} \frac{1}{V} \omega^n &= \log\int_Xe^{\beta(P(u)^\beta_t-u_t)} \frac{1}{V} \frac{\omega^n}{\omega_{P(u)_t^\beta}^n} \omega_{P(u)_t^\beta}^n\\
&\geq \frac{\beta}{V}\int_X(P(u)^\beta_t-u_t)\omega^n_{P(u)^\beta_t}-\mathrm{Ent}(P(u)^\beta_t).
\end{flalign*}
Putting these estimates together and letting $t \to \infty$, the lemma follows.
\end{proof}

\begin{lemma}\label{lem: beta_beta'cont} Let $0<D<\beta<\beta'$ and $u\in\mathcal E^1$ with $u\leq0$. Then there exists an increasing function $f_D:\mathbb R_+\to\mathbb R_+$ with $f_D(0)=0$ such that $f_D\in C^{0,1/2^n}_{\mathrm{loc}}(\mathbb R_+)$ and
$$d_1(u^\beta,u^{\beta'})\leq f_D(\beta'-\beta)(d_1(0,u)+1).$$
\end{lemma}

\begin{proof}
Choosing $\delta = 1 - \frac{\beta}{\beta'}$ in \cite[Lemma 6.7]{DZ25} we obtain:
    $$
u^\beta \leq \frac{\beta}{\beta'}u^\beta\leq \Big(\frac{\beta}{\beta'}u\Big)^{\beta'}-\frac{n}{\beta'}\log \Big(\frac{\beta}{\beta'}\Big).
$$
Using the triangle inequality, this estimate allows to initiate the following sequence of inequalities:
\begin{flalign*}
d_1&(u^\beta,u^{\beta'}) \leq d_1(u^\beta,\Big(\frac{\beta}{\beta'}u\Big)^{\beta'}-\frac{n}{\beta'}\log \Big(\frac{\beta}{\beta'}\Big)) + d_1(\Big(\frac{\beta}{\beta'}u\Big)^{\beta'}-\frac{n}{\beta'}\log \Big(\frac{\beta}{\beta'}\Big),u^{\beta'})\\
&\leq I(\Big(\frac{\beta}{\beta'}u\Big)^{\beta'})-I(u^\beta)  + d_1(\Big(\frac{\beta}{\beta'}u\Big)^{\beta'},u^{\beta'}) + \frac{2n}{\beta'}\log \Big(\frac{\beta'}{\beta}\Big)\\
&= I(u^{\beta'}) -I(u^{\beta}) + I(\Big(\frac{\beta}{\beta'}u\Big)^{\beta'}) - I(u^{\beta'})  + d_1(\Big(\frac{\beta}{\beta'}u\Big)^{\beta'},u^{\beta'}) + \frac{2n}{\beta'}\log \Big(\frac{\beta'}{\beta}\Big)\\
&\leq 2 d_1(\Big(\frac{\beta}{\beta'}u\Big)^{\beta'},u^{\beta'}) + \frac{2n}{\beta'}\log \Big(\frac{\beta'}{\beta}\Big),
\end{flalign*}
where in the last line we have used that $I$ is $1$-Lipschitz with respect to $d_1$, $\beta\mapsto I(u^\beta)$ is decreasing, hence $I(u^{\beta'})-I(u^\beta)\leq 0$. Using  \cite[Proposition 5.1]{DZ25} and the fact that $\log(1+t) \leq t$, we obtain that 
\begin{flalign*}
d_1&(u^\beta,u^{\beta'}) \leq C_n (1 - \frac{\beta}{\beta'})^{\frac{1}{2^n}}d_1(0,u) + \frac{2n}{\beta'} (\frac{\beta'}{\beta} -1).
\end{flalign*}
Taking $f_D(t):= C_D(t + t^{1/2^n}),$ the result follows.
\end{proof}

\begin{lemma}\label{lem: Ent_Holder} Suppose that $u\in\mathcal E^1$ with $\sup_Xu=0$. Let $D<\beta<\beta'$. Then there exists $C:=C(n,D)>0$ such that
\begin{flalign}\label{eq: ent_continuous}
|\frac{1}{\beta}\mathrm{Ent}(u^\beta)-\frac{1}{\beta'}\mathrm{Ent}(u^{\beta'})| &\leq C (f_D(\beta' - \beta)(d_1(0,u) + 1))^{\frac{1}{2^n}} \cdot |I(u)|^{1 - \frac{1}{2^n}},
\end{flalign}
where $f_D$ is the function from the previous result. 

In particular, $g(\varepsilon):=I(u^{1/\varepsilon})\in C^{1,1/4^n}(0,\infty)$ with $g'(\varepsilon)=\mathrm{Ent}(u^{1/\varepsilon})\in C^{0,1/4^n}(0,\infty)$.
\end{lemma}
\begin{proof} We start with the following computation:
\begin{flalign*}
\frac{V}{\beta}\mathrm{Ent}(u^\beta)-\frac{V}{\beta'}\mathrm{Ent}(u^{\beta'})
&=\int_X(u^\beta-u)\omega_{u^\beta}^n-\int_X(u^{\beta'}-u)\omega_{u^{\beta'}}^n\\
&=\int_X(u^\beta-u^{\beta'})\omega_{u^\beta}^n
+\int_Xu^{\beta'}(\omega_{u^\beta}^n-\omega_{u^{\beta'}}^n)
+\int_Xu(\omega_{u^{\beta'}}^n-\omega_{u^\beta}^n).
\end{flalign*}
After integrating by parts in the second and third integrals, \cite[Lemma A.2]{BBJ18} and \cite[Lemma 3.12]{DZ25} imply
\begin{flalign*}
\left|\frac{1}{\beta}\mathrm{Ent}(u^\beta)-\frac{1}{\beta'}\mathrm{Ent}(u^{\beta'})\right|
&\leq C d_1(u^\beta,u^{\beta'})^\frac{1}{2^n} |I(u)|^{1-\frac{1}{2^n}}.
\end{flalign*}
Applying the previous lemma to $d_1(u^\beta,u^{\beta'})$ gives \eqref{eq: ent_continuous}.

Regarding the last statement, Theorem \ref{thm:main_f'} gives
$g'(\varepsilon)=\mathrm{Ent}(u^{1/\varepsilon}).$
The $C^{0,1/4^n}$-regularity of the derivative follows from \eqref{eq: ent_continuous}, while the concavity follows from the fact that $\varepsilon\mapsto\mathrm{Ent}(u^{1/\varepsilon})$ is decreasing by \cite{DZ25}.
\end{proof}

The next proposition gives the most critical ingredients for the proof of Theorem \ref{thm: intro_slope}.

\begin{proposition}\label{prop: thm12cor} Let $\{u_t\} \in \mathcal R^1$ with $\sup_X u_t = 0$. Define
\[
g_t(\varepsilon) := \frac{I\big(u_t^{1/\varepsilon}\big)}{t}.
\]
Then each $g_t\in C^{1,1/4^n}(0,\infty)$ is concave, and its $C^{1,1/4^n}$ Hölder norm on compact subsets of $(0,\infty)$ is independent of $t>1$. Set
\[
g_\infty(\varepsilon):=I\big\{P(u)_t^{1/\varepsilon}\big\}\quad(\varepsilon>0),
\qquad g_\infty(0):=I\{u_t\}.
\]
We have $g_\infty\in C^{1,1/4^n}(0,\infty)$, and the following convergences hold uniformly on compact subsets of $(0,\infty)$:
$$g_t(\varepsilon)=\frac{I(u_t^{\frac{1}{\varepsilon}})}{t} \to I\{P(u)_t^{\frac{1}{\varepsilon}}\} = g_\infty(\varepsilon), \ \ \textup{ as } t \to \infty,$$
$$\frac{d}{d\varepsilon}g_t(\varepsilon)=\frac{\mathrm{Ent}(u_t^{1/\varepsilon})}{t}\to\frac{d}{d\varepsilon}g_\infty(\varepsilon)\quad\textup{as }t\to\infty.$$
\end{proposition}
\begin{proof} That $g_t(\varepsilon)=\frac{I(u_t^{\frac{1}{\varepsilon}})}{t} \to I\{P(u)_t^{\frac{1}{\varepsilon}}\} = g_\infty(\varepsilon)$, as $t \to \infty$, follows from \cite[Remark 6.9, Remark 6.10]{DZ25} and Theorem \ref{thm: maximizer_main}(ii). 

By \eqref{eq: ent_continuous} the $C^{1, 1/4^n}$-norms of each $g_t$ are uniformly bounded on any compact subset of $(0,\infty)$. Thus, the rest of the statement follows from the Arzela-Ascoli theorem of  $C^{1,1/4^n}$-H\"older functions.
\end{proof}

\begin{corollary}
If $\mathrm{Ent}\{u_t\}<\infty$, then
$$
\frac{d}{d\varepsilon}\bigg|_{\varepsilon=0^+}g_\infty(\varepsilon)=\mathrm{Ent}\{u_t\}.
$$
\end{corollary}

\begin{proof}
By the preceding energy-slope compatibility and the definition of radial quantized entropy,
\[
\frac{g_\infty(\varepsilon)-g_\infty(0)}{\varepsilon}
=\frac{I\{u_t^{1/\varepsilon}\}-I\{u_t\}}{\varepsilon}
=\mathrm{Ent}^{1/\varepsilon}\{u_t\}.
\]
The right-hand side converges to $\mathrm{Ent}\{u_t\}$ by \cite[Theorem 6.3]{DZ25}. This proves the asserted right derivative. It also shows that $g_\infty(\varepsilon)\to g_\infty(0)$; hence the concavity of $g_\infty$ on $(0,\infty)$ extends to $[0,\infty)$.
\end{proof}

To finish the proof of Theorem \ref{thm:intro_radial} we need to show that $\frac{d}{d\varepsilon}g_\infty(\varepsilon)=\mathrm{Ent}\{P(u)_t^{1/\varepsilon}\}$ for any $\varepsilon>0$. By the previous discussion we already know that this derivative exists, but we need to determine its exact value. We start with a special case.\medskip

\begin{lemma} Suppose that $\{u_t\} \in \mathcal R^1$ is induced by test configurations. Then
\begin{equation}\label{eq: to_prove}
\frac{d}{d\varepsilon}g_\infty(\varepsilon)=\mathrm{Ent}\{P(u)_t^{1/\varepsilon}\},\qquad\varepsilon>0.
\end{equation}
\end{lemma}
\begin{proof}
After pulling back to a smooth dominating model, let $(\mathcal X,\mathcal L)$ be a semi-ample test configuration dominating
$X\times\mathbb P^1$ that induces $\{u_t\}$. Set
$K^{\log}:=K^{\log}_{\mathcal X/X\times\mathbb P^1}$ and
$\mathcal L_\varepsilon:=\mathcal L+\varepsilon K^{\log}$. By
\cite[Remark 9.2]{DZ25}, the Li type model
$(\mathcal X,\mathcal L_\varepsilon)$ induces the ray
$\{P(u)_t^{1/\varepsilon}\}$. Choose $c\gg 1$ so that
$\mathcal L+c\mathcal X_0+\varepsilon K^{\log}$ is big for the values of
$\varepsilon$ under consideration. The Deligne functional slope formula for
models from \cite{Li20-cscK} gives
\[
 g_\infty(\varepsilon)=I\{P(u)_t^{1/\varepsilon}\}
 =\frac{\vol(\mathcal L+c\mathcal X_0+\varepsilon K^{\log})}{(n+1)V}-c .
\]
In the semi-ample case this is the usual intersection calculation, and in the
big case it follows from the differentiability of the volume function. Hence
\[
 \frac{d}{d\varepsilon}g_\infty(\varepsilon)
 =\frac{1}{V}\big\langle(\mathcal L+c\mathcal X_0+
 \varepsilon K^{\log})^n\big\rangle\cdot K^{\log}.
\]
By Li's intersection formula for log discrepancy models \cite{Li20-cscK}, or equivalently by the
non-Archimedean entropy formula, the last expression is
\[
 \int_{X^{an}} A_X\,\mathrm{MA}^{\mathrm{NA}}\{P(u)^{1/\varepsilon}\}
 =\mathrm{Ent}\{P(u)_t^{1/\varepsilon}\}.
\]
This proves \eqref{eq: to_prove}.
\end{proof}

\begin{lemma}\label{lem: P_beta_rad_conv} Suppose that $\{u_t^j\},\{u_t\} \in \mathcal R^1$ and $\varepsilon>0$. If we have that $d_1^c(\{u_t^j\},\{u_t\}) \to 0$ then $d_1^c(\{P(u^j)_t^{1/\varepsilon}\},\{P(u)_t^{1/\varepsilon}\}) \to 0$.
\end{lemma}

\begin{proof}
Using \cite[Proposition 5.1]{DZ25}, we obtain
\[
\frac{1}{t}d_1(u_t^{j,1/\varepsilon},u_t^{1/\varepsilon})
\leq C
\left(\frac{d_1(u_t^j,u_t)}{t}\right)^{\frac{1}{2^n}}
\max\left\{
\frac{d_1(0,u_t^j)}{t},
\frac{d_1(0,u_t)}{t}
\right\}^{1-\frac{1}{2^n}}.
\]
Taking the limsup as $t\to\infty$ gives
\[
\limsup_{t\to\infty}
\frac{d_1(u_t^{j,1/\varepsilon},u_t^{1/\varepsilon})}{t}
\leq
C\,d_1^c(\{u_t^j\},\{u_t\})^{\frac{1}{2^n}}
\max\left\{
d_1^c(\{0\},\{u_t^j\}),
d_1^c(\{0\},\{u_t\})
\right\}^{1-\frac{1}{2^n}}.
\]
The quantities $d_1^c(\{0\},\{u_t^j\})$ are uniformly bounded, since
\[
d_1^c(\{0\},\{u_t^j\})
\leq
d_1^c(\{0\},\{u_t\})
+d_1^c(\{u_t^j\},\{u_t\}),
\]
and $d_1^c(\{u_t^j\},\{u_t\})\to0$.

By \cite[Remark 6.10]{DZ25} and Theorem \ref{thm: maximizer_main}(ii),
$\{P(u^j)_t^{1/\varepsilon}\}$ and $\{P(u)_t^{1/\varepsilon}\}$ are the
geodesic rays asymptotic to $\{u_t^{j,1/\varepsilon}\}$ and
$\{u_t^{1/\varepsilon}\}$, respectively. In other words,
\[
\lim_{t\to\infty}
\frac{d_1(P(u^j)_t^{1/\varepsilon},u_t^{j,1/\varepsilon})}{t}=0,
\qquad
\lim_{t\to\infty}
\frac{d_1(P(u)_t^{1/\varepsilon},u_t^{1/\varepsilon})}{t}=0.
\]
The triangle inequality therefore yields
\[
d_1^c(\{P(u^j)_t^{1/\varepsilon}\},\{P(u)_t^{1/\varepsilon}\})
\leq
C\,d_1^c(\{u_t^j\},\{u_t\})^{\frac{1}{2^n}}
\max\left\{
d_1^c(\{0\},\{u_t^j\}),
d_1^c(\{0\},\{u_t\})
\right\}^{1-\frac{1}{2^n}},
\]
which tends to zero by assumption.
\end{proof}

\begin{corollary}
Let $\{u_t\}\in\cR^1$ be a maximal ray.
Suppose that $\{u_t^j\} \in \mathcal R^1$ is the Demailly approximation of $\{u_t\}$ and $\varepsilon>0$. Let $g^j_\infty(\varepsilon) := I\{P(u^j)^{1/\varepsilon}_t\}$ and $g_\infty(\varepsilon) := I\{P(u)^{1/\varepsilon}_t\}$. Then for any $\varepsilon >0$ we have
$$g^j_\infty(\varepsilon) \to g_\infty(\varepsilon),$$
$$\frac{d}{d\varepsilon}g^j_\infty(\varepsilon)=\mathrm{Ent}\{P(u^j)_t^{1/\varepsilon}\}\to\frac{d}{d\varepsilon}g_\infty(\varepsilon).$$
\end{corollary}

\begin{proof}
That $g^j_\infty(\varepsilon)\to g_\infty(\varepsilon)$ follows from the previous
result and $d_1^c$-Lipschitz continuity of $I\{\cdot\}$. For the convergence of the derivatives we use the following standard trick going back to \cite{BBN}. Let $f_j,f:(0,\infty)\to\mathbb R$ be concave $C^1$ functions such that
$f_j\to f$ pointwise, and assume that $f$ is differentiable at $\varepsilon$.
For $0<h<\varepsilon$, concavity gives
\[
 \frac{f_j(\varepsilon)-f_j(\varepsilon-h)}{h}
 \geq f_j'(\varepsilon)\geq
 \frac{f_j(\varepsilon+h)-f_j(\varepsilon)}{h}.
\]
Taking $\limsup_j$ and $\liminf_j$, and then letting $h\downarrow 0$, gives
\[
 \limsup_{j\to\infty} f_j'(\varepsilon)\leq f'(\varepsilon)
 \leq \liminf_{j\to\infty} f_j'(\varepsilon).
\]
Thus $f_j'(\varepsilon)\to f'(\varepsilon)$. Applying this to
$f_j=g^j_\infty$ and $f=g_\infty$, and using the preceding test configuration
case for the identity
$(g^j_\infty)'(\varepsilon)=\mathrm{Ent}\{P(u^j)_t^{1/\varepsilon}\}$, proves the claim.
\end{proof}

\begin{lemma}\label{lem: Ent_P_beta_conv} 
Let $\{u_t\}\in\cR^1$ be a maximal ray. Then 
$\mathrm{Ent}\{P(u^j)_t^{1/\varepsilon}\}\to
\mathrm{Ent}\{P(u)_t^{1/\varepsilon}\}$, when $\{u^j_t\}\in\mathcal R^1$ is
the Demailly approximation of a maximal ray $\{u_t\} \in \mathcal R^1$. So we have 
$$
\frac{d}{d\varepsilon}g_\infty(\varepsilon)=\mathrm{Ent}\{P(u)_t^{1/\varepsilon}\},\qquad\varepsilon>0.
$$
\end{lemma}
\begin{proof}
Using Lemma \ref{lem: entropy_est} we can start the following estimates:
\begin{flalign}\label{eq: to_prov_ineq}
\mathrm{Ent}\{P(u)^{1/\varepsilon}_t\} & \geq  \limsup_{t \to \infty}\frac{1}{\varepsilon tV}\int_{X} (P(u)_t^{1/\varepsilon} - u_t)  \omega^n_{P(u)_t^{1/\varepsilon}}\\
&= \limsup_{j \to \infty} \limsup_{t \to \infty}\frac{1}{\varepsilon tV}\int_{X} (P(u^j)_t^{1/\varepsilon} - u^j_t)  \omega^n_{P(u^j)_t^{1/\varepsilon}} \nonumber\\
&= \limsup_{j \to \infty} \lim_{t \to \infty}\frac{1}{\varepsilon tV}\int_{X} (P(u^j)_t^{1/\varepsilon} - u^j_t)  \omega^n_{P(u^j)_t^{1/\varepsilon}} \nonumber \\
&= \limsup_{j \to \infty} \frac{1}{\varepsilon}\int_{X^{an}} (P(u^j)^{1/\varepsilon} - u^j)\mathrm{MA}^{\mathrm{NA}}\{P(u^j)^{1/\varepsilon}\} \nonumber \\
&= \limsup_{j \to \infty} \int_{X^{an}} A_X\,\mathrm{MA}^{\mathrm{NA}}\{P(u^j)^{1/\varepsilon}\} \nonumber\\
&= \limsup_{j \to \infty}\mathrm{Ent}\{P(u^j)^{1/\varepsilon}_t\}. \nonumber
\end{flalign}
In the second line we used Lemma \ref{lem: P_beta_rad_conv} together with the
standard continuity estimate for mixed Monge--Amp\`ere energies under
$d_1^c$-convergence. The third line uses that $\{u_t^j\}$ is induced by a
Demailly approximation, hence by a test configuration. The fourth line is the
Deligne pairing slope formula for the functional
$\int_X(\cdot-\cdot)\omega^n_{\cdot}$, as in \cite{SD18}. Finally, if
$U^j$ denotes the non-Archimedean potential of $\{u_t^j\}$, then
$P(u^j)^{1/\varepsilon}$ is the envelope $P(U^j+\varepsilon A_X)$. The
orthogonality relation for this envelope from \cite{Li21} gives
\[
 \int_{X^{an}}(P(u^j)^{1/\varepsilon}-u^j)\mathrm{MA}^{\mathrm{NA}}\{P(u^j)^{1/\varepsilon}\}
 =\varepsilon\int_{X^{an}}A_X\,\mathrm{MA}^{\mathrm{NA}}\{P(u^j)^{1/\varepsilon}\}.
\]
This justifies the fourth and fifth lines, while the last line is the
non-Archimedean/radial entropy formula for test configurations.

On the other hand, Lemma \ref{lem: P_beta_rad_conv} gives
$d_1^c(\{P(u^j)_t^{1/\varepsilon}\},\{P(u)_t^{1/\varepsilon}\})\to0$, and the
radial entropy is $d_1^c$-lower semicontinuous. Hence
\[
 \mathrm{Ent}\{P(u)_t^{1/\varepsilon}\}\leq
 \liminf_{j\to\infty}\mathrm{Ent}\{P(u^j)_t^{1/\varepsilon}\}.
\]
Together with \eqref{eq: to_prov_ineq}, this proves the desired convergence.
\end{proof}

\begin{corollary}Let $\{u_t\}\in\cR^1$ be a maximal ray and $\varepsilon >0$. Then
\begin{equation}\label{eq: entr_formula} 
\mathrm{Ent}\{P(u)^{1/\varepsilon}_t\}=\lim_{t\to\infty}\frac{1}{\varepsilon tV}\int_X(P(u)_t^{1/\varepsilon}-u_t)\omega^n_{P(u)_t^{1/\varepsilon}}.
\end{equation}
\end{corollary}
\begin{proof}As a consequence of the previous argument, we must have equality in \eqref{eq: to_prov_ineq}, yielding the corollary, with limsup on the right hand side. Since in case of test configurations using the Deligne formalism the limsup is a limit \cite{SD18}, the general result follows using approximation.
\end{proof}

The following lemma finishes the proof of Theorem \ref{thm:intro_radial}, showing that $H_u \in C^1([0,\infty))$:

\begin{lemma}
    For any $\{u_t\}\in\cR^1$ with $\mathrm{Ent}\{u_t\}<\infty$, one has
    $$
    \mathrm{Ent}\{u_t\}=\lim_{\varepsilon\to0^+}\mathrm{Ent}\{P(u)_t^{1/\varepsilon}\}.
    $$
\end{lemma}

\begin{proof}
Since the continuous extension of $g_\infty$ is concave on $[0,\infty)$, we have
\[
\mathrm{Ent}\{u_t\}=g'_{\infty,+}(0)\geq g'_\infty(\varepsilon)
=\mathrm{Ent}\{P(u)_t^{1/\varepsilon}\}.
\]
Moreover, $P(u)_t^{1/\varepsilon}\geq u_t$, so the ordered-potential formula for $d_1$ and energy-slope compatibility give
\begin{align*}
d_1^c(\{u_t\},\{P(u)_t^{1/\varepsilon}\})
&=I\{P(u)_t^{1/\varepsilon}\}-I\{u_t\}\\
&=I\{u_t^{1/\varepsilon}\}-I\{u_t\}
=\varepsilon\,\mathrm{Ent}^{1/\varepsilon}\{u_t\}
\leq\varepsilon\,\mathrm{Ent}\{u_t\}\longrightarrow0.
\end{align*}
The $d_1^c$-lower semicontinuity of radial entropy therefore yields
\[
\mathrm{Ent}\{u_t\}\leq\liminf_{\varepsilon\to0^+}
\mathrm{Ent}\{P(u)_t^{1/\varepsilon}\}.
\]
Combining the last two inequalities proves the claim.
\end{proof}

The following result will be needed in our applications to canonical metrics:

\begin{corollary}\label{cor: Ent_id} Let $\{u_t\}\in\mathcal R^1$ with $\sup_X u_t =0$ such that 
$\mathrm{Ent}\{u_t\}=\mathrm{Ent}\{P(u)^{1/\varepsilon}_t\}<\infty$ for some $\varepsilon>0$. Then $\{P(u)^{1/\varepsilon}_t\}=\{u_t\}$.
\end{corollary}
\begin{proof}
By $\varepsilon$-concavity and \eqref{eq: entr_formula} we have
\begin{equation}\label{eq: Ent_ineq}
\mathrm{Ent}\{u_t\}\geq\frac{I\{P(u)^{1/\varepsilon}_t\}-I\{u_t\}}{\varepsilon}
\geq\mathrm{Ent}\{P(u)^{1/\varepsilon}_t\} = \lim_{t\to\infty}\frac{1}{\varepsilon tV}\int_X(P(u)_t^{1/\varepsilon}-u_t)\omega^n_{P(u)_t^{1/\varepsilon}}.
\end{equation}
By the hypothesis, equality holds throughout. Let
\[
\mathcal I(p,q):=\frac{1}{V}\int_X(p-q)(\omega_q^n-\omega_p^n).
\]
Since $I(p) - I(q) - \int_X (p-q) \omega_p^n \simeq \mathcal I(p,q)$, the above equalities give that 
\[
\lim_{t\to\infty}\frac{\mathcal I(P(u)_t^{1/\varepsilon},u_t)}{t}=0.
\]
Using the homogeneous estimate of \cite[Lemma A.1]{BBJ18}, the above gives that 
$$\limsup_{t \to \infty }\frac{I(P(u)_t^{1/\varepsilon}) - I(u_t) - \int_X (P(u)_t^{1/\varepsilon} - u_t) \omega^n}{t} =0.$$
We have $\int_X (P(u)_t^{1/\varepsilon} - u_t) \omega^n/t \to 0$, since $\sup_X u_t = \sup_X P(u)_t^{1/\varepsilon} =0$. Thus $I\{P(u)_t^{1/\varepsilon}\} = I(u_t)$. Since $P(u)_t^{1/\varepsilon} \geq u_t$, we obtain 
$\{P(u)^{1/\varepsilon}_t\}=\{u_t\}$.
\end{proof}

\section{Non-Archimedean entropy}
\label{sec: NA}

In this part we prove Theorem \ref{thm: main_F'} and Theorem \ref{thm: main_reg_NA_ent}. We start with a general discussions in NA geometry.

Let $(X,L)$ be a polarized complex projective manifold of dimension $n$. Let $\mathrm{PSH}^{\mathrm{NA}}$ be the set of $L$-psh functions on $X^{an}$.
Given any $U\in \mathrm{PSH}^{\mathrm{NA}}$ and $\varepsilon\geq 0$ let us investigate the following envelope:
\begin{equation}
    \label{eq: def_U_e}
    U_\varepsilon:=\mathrm{P}(U+\varepsilon A_X):=\mathrm{usc}\sup \left\{V:V\in \mathrm{PSH}^{\mathrm{NA}},\ V\leq U+\varepsilon A_X\text{ on }X^{an}\right\}.
\end{equation}
Here $\mathrm{usc}$ denotes the upper semi-continuous regularization (which turns out to be unnecessary, by Lemma \ref{lem: Ue<U+eA} below).
To the best of our knowledge, such a consideration first appeared in \cite[\S 7]{X23}. 

It is clear that
$
U\leq U_\varepsilon
$
and $U=U_0$. Now we establish several basic properties of $U_\varepsilon.$

\begin{lemma}
For any $U\in \mathrm{PSH}^{\mathrm{NA}}$ and $\varepsilon\geq 0$,
    one has $U_\varepsilon\in \mathrm{PSH}^{\mathrm{NA}}$ and
    \begin{equation}
        \label{eq: sup_Ue=sup_U}
         \sup U_\varepsilon=\sup U.
    \end{equation}
\end{lemma}

\begin{proof}

For any candidate $V\in \mathrm{PSH}^{\mathrm{NA}}$ of the envelope $U_\varepsilon$, we have (by \cite[(6.1)]{BBJ18})
    $$
    \sup V=V(v_{triv})\leq U(v_{triv})+\varepsilon A_X(v_{triv})=\sup U.
    $$
    Thus
    $U_\varepsilon\in \mathrm{PSH}^{\mathrm{NA}}$ due to \cite[Corollary 5.28]{BoJ18a}, and $\sup U_\varepsilon\leq\sup U$. We then obtain \eqref{eq: sup_Ue=sup_U} since one also has $U\leq U_\varepsilon$.
\end{proof}

\begin{lemma}
\label{lem: Ue<U+eA}
    One has $U_\varepsilon\leq U+\varepsilon A_X$, so the upper semi-continuous regularization in \eqref{eq: def_U_e} is unnecessary.
\end{lemma}

\begin{proof}
    Pick a decreasing net $\{U^k\}\subset\cH^{\mathrm{NA}}$ converging to $U$. Then each $U^k$ is determined by some smooth test configuration $(\cX^k,\cL^k)$ dominating $X\times\PP^1$. Now we use a crucial identity due to Xia \cite[(7.2)]{X23}:
    \begin{equation}
        \label{eq: xia}
         U^k_\varepsilon=P(U^k+\varepsilon A_X)=P(U^k+\varepsilon A_X\circ r_{\cX^k}).
    \end{equation}
    Recall from \S \ref{sub: NA_pre} that $r_{\cX^k}$ denotes the retraction map from $X^{an}$ to the dual complex $\Delta_{\cX^k}$ of $\cX^k$. Since both $U^k$ and $A_X\circ r_{\cX^k}$ are continuous on $X^{an}$, we deduce that
    \begin{equation}
        \label{eq: Uke<U+eAr}
        U_\varepsilon\leq U^k_\varepsilon\leq U^k+\varepsilon A_X\circ r_{\cX^k}\leq U^k+\varepsilon A_X,
    \end{equation}
    where we used the fact $A_X\circ r_{\cX^k}\leq A_X$ (see \cite[Appendix A]{BoJ18}). Passing $k$ to the limit, we obtain that $U_\varepsilon\leq U+\varepsilon A_X$, finishing the proof.
\end{proof}

The next result also follows easily from the above proof.
\begin{corollary}
\label{cor:Ue_to_U}
    For any $U\in \mathrm{PSH}^{\mathrm{NA}}$, one has $U_\varepsilon\searrow U$ as $\varepsilon\searrow 0$.
\end{corollary}

\begin{proof}
It is clear that $U\leq U_\varepsilon$ for all $\varepsilon>0$. So we deduce from \cite[Proposition 5.6(vi)]{BoJ18a} that the decreasing limit $V:=\lim_{\varepsilon\searrow 0}U_\varepsilon$ belongs to $\mathrm{PSH}^{\mathrm{NA}}$. It is evident that $U\leq V$. As in \eqref{eq: Uke<U+eAr}, one can also write
$$
V\leq U^k+\varepsilon A_X\circ r_{\cX^k}\text{ for any }\varepsilon>0\text{ and }k.
$$
Sending $\varepsilon\to 0$ and then passing $k$ to the limit, we obtain $V\leq U$ and conclude.
\end{proof}

In what follows we are only interested in the case when $U\in\cE^{1,\mathrm{NA}}$ and we will argue that $U_\varepsilon$ is the NA analogue of the solution $u_\varepsilon$ of \eqref{eq: def_u_e}.
To fully appreciate this analogy,
we start with a NA version of Ko\l odziej's $L^\infty$ estimate \cite{Kol98} for $U_\varepsilon$.
\begin{proposition}
\label{prop:U_e_bd}
For any $U\in\cE^{1,\mathrm{NA}}$ and $\varepsilon>0$, $U_\varepsilon$ is a bounded $L$-psh function.
\end{proposition}

\begin{proof}
    We can assume that $\sup U=0$. In view of \eqref{eq: sup_Ue=sup_U}, it suffices to show that $U_\varepsilon$ is bounded from below. 
    
    By \cite[Theorem 6.6]{BBJ18} we have a geodesic ray $\{u_t\}$ with $\sup u_t=0$, whose associated NA potential is $U$. Then applying the multiplier ideal sheaf regularization procedure as in the proof of \cite[Theorem 6.2]{BBJ18}, we obtain a decreasing sequence
    $\cH^{\mathrm{NA}}\ni U^k\searrow U$ and (by \cite[Lemma 5.7]{BBJ18})
    $$
    (2^k+m_0)U^k\leq 2^k U+A_X+1 \text{ on }X^{div},
    $$
    where $m_0>1$ is some large integer (only depending on $X,L$). Rearranging, we get
    $$
    (2^k+m_0)U^k-2^k U-1\leq A_X \text{ on }X^{div}.
    $$
    Recall that $U^{k}\in\cH^{\mathrm{NA}}$ is continuous on $X^{an}$. Thus, both sides of the above inequality are lsc functions on $X^{an}$. According to \cite[Appendix A]{BoJ18}, $A_X: X^{an}\to [0,+\infty]$ is the largest lsc extension of $A_X:X^{div}\to [0,+\infty)$. So we have
    $$
    (2^k+m_0)U^k-2^k U-1\leq A_X \text{ on }X^{an}.
    $$
    In other words,
    \begin{equation}
        \label{eq: U_k<U+A}
        (2^k+m_0)U^k\leq 2^k U+A_X+1 \text{ on }X^{an}.
    \end{equation}

  Now, for any $\varepsilon>0$, we choose some $k_\varepsilon\geq 1$ such that $1/2^{k_\varepsilon}<\varepsilon$. Then we have
  \begin{equation*}
      \label{eq:U_k<U+A+1}
       (1+\frac{m_0}{2^{k_\varepsilon}})U^{k_\varepsilon}-\frac{1}{2^{k_\varepsilon}}\leq U+\frac{1}{2^{k_\varepsilon}}A_X\leq U+\varepsilon A_X\text{ on }X^{an}.
  \end{equation*}
  Since $U^{k_\varepsilon}\in\cH^{\mathrm{NA}}$ is continuous on $X^{an}$, we can find some constant $C_\varepsilon>0$ such that
  \begin{equation}
      \label{eq:U+A>_C}
      U+\varepsilon A_X\geq -C_\varepsilon\text{ on }X^{an}.
  \end{equation}
  Consequently, $U_\varepsilon\geq-C_\varepsilon$, finishing the proof.
\end{proof}

We remark that \eqref{eq:U+A>_C} should be viewed as the NA version of the fact that any $u\in\cE^1$ has zero Lelong number \cite{GZ07}.

The next result further strengthens Proposition \ref{prop:U_e_bd}, which is a NA version of Ko\l odziej's continuity estimate \cite{Kol98} for $u_\varepsilon$.
\begin{proposition}
\label{prop: U_e_conti}
  For any $U\in\cE^{1,\mathrm{NA}}$ and $\varepsilon>0$, $U_\varepsilon$ is a continuous $L$-psh function on $X^{an}$.
\end{proposition}

\begin{remark}
\label{rem: cont_NApsh_induced_fromfilt}
    By \cite[Theorem B]{BoJ18}, any continuous $L$-psh function is lower regularizable and is induced by some filtration of $R(X,L)$. In fact, one can explicitly write down the filtration associated to $U_\varepsilon$ using the graded supremum norm construction in \cite[Corollary 4.11]{BoJ18}.
\end{remark}

\begin{proof}
We can assume that $\sup U=0$.
    Let $\{U^k\}\subset\cH^{\mathrm{NA}}$ be the regularization sequence as in the proof of Proposition \ref{prop:U_e_bd}. For any $\varepsilon\geq 0$, set
    $$
    U^k_\varepsilon:=P(U^k+\varepsilon A_X).
    $$
Recalling \eqref{eq: xia} and \cite[Lemma 6.31]{BoJ18a}, we know that $U^k_\varepsilon$ is continuous on $X^{an}$.
It is also clear that $U_\varepsilon^k$ is decreasing in $k$ and $U_\varepsilon^k\geq U_\varepsilon$.
So it remains to show that $U^k_\varepsilon\searrow U_\varepsilon$ uniformly.

To this end, we fix $\varepsilon>0$. For any $V\in \mathrm{PSH}^{\mathrm{NA}}$ such that $V\leq U^k+\varepsilon A_X$, we can write (using \eqref{eq: U_k<U+A})
\begin{equation*}
    \begin{split}
        V&\leq U^k+\varepsilon A_X\leq \frac{2^k}{2^k+m_0} U+\left(\varepsilon+\frac{1}{2^k+m_0}\right)A_X+\frac{1}{2^k+m_0}\\
        &=U+\left(\varepsilon+\frac{1+m_0}{2^k+m_0}\right)A_X-\frac{m_0}{2^k+m_0}(U+A_X)+\frac{1}{2^k+m_0}\\
        &\leq U+\left(\varepsilon+\frac{1+m_0}{2^k+m_0}\right)A_X+\frac{C_1m_0+1}{2^k+m_0}\leq U+\left(\varepsilon+\frac{C_1m_0+1}{2^k+m_0}\right)A_X+\frac{C_1m_0+1}{2^k+m_0},
    \end{split}
\end{equation*}
where in the last line  we used that $U+A_X\geq -C_1$ on $X^{an}$ for some $C_1>1$ (recall \eqref{eq:U+A>_C}). Let us put for simplicity that
$$
\delta_k:=\frac{C_1m_0+1}{2^k+m_0}.
$$
Then $\delta_k\to 0$ as $k\to\infty$, and the above analysis implies that
\begin{equation}
    \label{eq: U_k_e<U_e+d+d}
    U^k_\varepsilon\leq U_{\varepsilon+\delta_k}+\delta_k \text{ for all }k\geq 1.
\end{equation}

Next let us control the difference $|U_\varepsilon-U_{\varepsilon+\delta_k}|$. It is clear that $U_\varepsilon\leq U_{\varepsilon+\delta_k}$. On the other hand, for any $V_1\in \mathrm{PSH}^{\mathrm{NA}}$ with $V_1\leq U+(\varepsilon+\delta_k)A_X$ and $V_2\in \mathrm{PSH}^{\mathrm{NA}}$ with $V_2\leq U+\frac{\varepsilon}{2}A_X$,
one has that
\begin{equation*}
    \begin{split}
        \left(1-  \frac{2\delta_k}{\varepsilon+2\delta_k}\right)&V_1+\frac{2\delta_k}{\varepsilon+2\delta_k}V_2\\
        &\leq U+\left(\left(1-\frac{2\delta_k}{\varepsilon+2\delta_k}\right)(\varepsilon+\delta_k)+\frac{2\delta_k}{\varepsilon+2\delta_k}\frac{\varepsilon}{2}\right)A_X=U+\varepsilon A_X.
    \end{split}
\end{equation*}
So $ \left(1-\frac{2\delta_k}{\varepsilon+2\delta_k}\right)V_1+\frac{2\delta_k}{\varepsilon+2\delta_k}V_2\in \mathrm{PSH}^{\mathrm{NA}}$ (\cite[Proposition 5.6(v)]{BoJ18a}) is a candidate for $U_\varepsilon$, implying that
$$
\left(1-\frac{2\delta_k}{\varepsilon+2\delta_k}\right)V_1+\frac{2\delta_k}{\varepsilon+2\delta_k}V_2\leq U_\varepsilon.
$$
Using that $\sup U_{\varepsilon+\delta_k}=\sup U=0$ (recall \eqref{eq: sup_Ue=sup_U}) and $U+\frac{\varepsilon}{2}A_X\geq-C_{\frac{\varepsilon}{2}}$ (recall \eqref{eq:U+A>_C}), we have $V_1\leq 0$ and can take $V_2:=-C_{\frac{\varepsilon}{2}}.$

Then we arrive at
$$
V_1\leq U_\varepsilon+\frac{2\delta_k C_{\frac{\varepsilon}{2}}}{\varepsilon+2\delta_k},
$$
so that
$$
U_{\varepsilon+\delta_k}\leq U_\varepsilon+\frac{2\delta_k C_{\frac{\varepsilon}{2}}}{\varepsilon+2\delta_k}.
$$
Combining \eqref{eq: U_k_e<U_e+d+d}, we deduce that
$$
U_\varepsilon\leq U_\varepsilon^k\leq U_\varepsilon+\frac{2\delta_k C_{\frac{\varepsilon}{2}}}{\varepsilon+2\delta_k}+\delta_k.
$$
So we find that the convergence $U^k_\varepsilon\searrow U_\varepsilon$ is uniform in $k$, finishing the proof.
\end{proof}

The following uniform estimate can be extracted from the above proof.

\begin{corollary}
\label{cor: U_e_Lip_in_e}
    For any $U\in\cE^{1,\mathrm{NA}}$, $\varepsilon>0$ and $\delta>0$, we have that
    $$
    |U_{\varepsilon+\delta}-U_\varepsilon|\leq B_\varepsilon\delta
    $$
   for some constant $B_\varepsilon>0$ depending on $X,L,\varepsilon$ and $U$.
\end{corollary}

The next result shall be compared with its `Archimedean' cousin \eqref{eq: ent_ue_int}.

\begin{proposition}
\label{prop:Ent_U_e=int_Ue_U}
    For any $U\in\cE^{1,\mathrm{NA}}$ and $\varepsilon>0$, one has
    $$
    \mathrm{Ent}^{\mathrm{NA}}(U_\varepsilon)=\frac{1}{\varepsilon}\int_{X^{an}}(U_\varepsilon-U)\mathrm{MA}^{\mathrm{NA}}(U_\varepsilon).
    $$
\end{proposition}

\begin{proof}
    Let us first treat the case when $U\in\cH^{\mathrm{NA}}$. Then $U$ is determined by some smooth test configuration $(\cX,\cL)$ dominating $X\times\PP^1$. So \eqref{eq: xia} gives
    $$
    U_\varepsilon=P(U+\varepsilon A_X\circ r_\cX).
    $$
    Since both $U$ and $A_X\circ r_\cX$ are continuous on $X^{an}$, by orthogonality \cite[Theorem 6.30(i)]{BoJ18a}, we can write
    \begin{equation*}
        \begin{split}
             \frac{1}{\varepsilon }\int_{X^{an}}(U_\varepsilon-U)\mathrm{MA}^{\mathrm{NA}}(U_\varepsilon)&=\frac{1}{\varepsilon}\int_{X^{an}}(P(U+\varepsilon A_X\circ r_\cX)-U)\mathrm{MA}^{\mathrm{NA}}(U_\varepsilon)\\
             &=\int_{X^{an}} A_X\circ r_\cX \mathrm{MA}^{\mathrm{NA}}(U_\varepsilon).
        \end{split}
    \end{equation*}
    We claim that for any $\varepsilon>0$ it holds that
    \begin{equation}
        \label{eq: orth}
        \mathrm{Ent}^{\mathrm{NA}}(U_\varepsilon)=\frac{1}{\varepsilon}\int_{X^{an}}(U_\varepsilon-U)\mathrm{MA}^{\mathrm{NA}}(U_\varepsilon)=\int_{X^{an}} A_X\circ r_\cX \mathrm{MA}^{\mathrm{NA}}(U_\varepsilon).
    \end{equation}

    Indeed, when $\varepsilon\in\QQ_{>0}$, note that $\mathrm{MA}^{\mathrm{NA}}(U_\varepsilon)$ is supported on $\Delta_\cX$, as $U_\varepsilon$ is the NA potential of the model $(\cX,\cL+\varepsilon K^{\log}_{\cX/X\times\PP^1})$ (see \cite[Remark 9.2]{DZ25} and \cite[(1)]{Li21}). This implies that
    \begin{equation}
    \label{eq: Ent_NA_Ue=int}
        \begin{split}
            \mathrm{Ent}^{\mathrm{NA}}(U_\varepsilon)&=\int_{X^{an}}A_X \mathrm{MA}^{\mathrm{NA}}(U_\varepsilon)=\int_{X^{an}}A_X\circ r_\cX \mathrm{MA}^{\mathrm{NA}}(U_\varepsilon)\\
            &=\frac{1}{\varepsilon }\int_{X^{an}}(U_\varepsilon-U)\mathrm{MA}^{\mathrm{NA}}(U_\varepsilon).\\
        \end{split}
    \end{equation}
    holds $\text{for }U\in\cH^{\mathrm{NA}}\text{ and }\varepsilon\in\QQ_{>0}.$

    Now for any $\varepsilon>0$, we pick a rational sequence $\varepsilon_i\searrow\varepsilon$, then Corollary \ref{cor: U_e_Lip_in_e} implies that $U_{\varepsilon_i}\searrow U_\varepsilon$, so in particular
    
    $$
    \mathrm{Ent}^{\mathrm{NA}}(U_\varepsilon)\leq\liminf_{i\to\infty}\mathrm{Ent}^{\mathrm{NA}}(U_{\varepsilon_i}).
    $$
    On the other hand, since $A_X\circ r_\cX$ is continuous and $\mathrm{MA}^{\mathrm{NA}}(U_{\varepsilon_i})\to \mathrm{MA}^{\mathrm{NA}}(U_\varepsilon)$ weakly,
    $$
    \mathrm{Ent}^{\mathrm{NA}}(U_\varepsilon)\geq\int_{X^{an}}A_X\circ r_\cX \mathrm{MA}^{\mathrm{NA}}(U_\varepsilon)=\lim_{i\to\infty}\int_{X^{an}}A_X\circ r_\cX \mathrm{MA}^{\mathrm{NA}}(U_{\varepsilon_i})=\lim_{i\to\infty}\mathrm{Ent}^{\mathrm{NA}}(U_{\varepsilon_i}),
    $$
    where we used $A_X\geq A_X\circ r_\cX$ (see \cite[Appendix A]{BoJ18}) and \eqref{eq: Ent_NA_Ue=int}.
    Consequently,
    \begin{equation*}
        \begin{split}
            \mathrm{Ent}^{\mathrm{NA}}(U_\varepsilon)&=\lim_{i\to\infty}\mathrm{Ent}^{\mathrm{NA}}(U_{\varepsilon_i})=\lim_{i\to\infty}\frac{1}{\varepsilon_i }\int_{X^{an}}(U_{\varepsilon_i}-U)\mathrm{MA}^{\mathrm{NA}}(U_{\varepsilon_i})\\
            &=\frac{1}{\varepsilon }\int_{X^{an}}(U_\varepsilon-U)\mathrm{MA}^{\mathrm{NA}}(U_\varepsilon)=\int_{X^{an}} A_X\circ r_\cX \mathrm{MA}^{\mathrm{NA}}(U_\varepsilon)
        \end{split}
    \end{equation*}
holds for any $U\in\cH^{\mathrm{NA}}$ and $\varepsilon>0$. Here we used that $U_{\varepsilon_i}\to U_\varepsilon$ uniformly (recall Corollary \ref{cor: U_e_Lip_in_e} and \cite[Corollary 6.14]{BoJ18a}).

Now we treat the general case.
For any $U\in\cE^{1,\mathrm{NA}}$ and $\varepsilon>0$ , Lemma \ref{lem: Ue<U+eA} implies that
$$
\frac{1}{\varepsilon }\int_{X^{an}}(U_\varepsilon-U)\mathrm{MA}^{\mathrm{NA}}(U_\varepsilon)\leq \int_{X^{an}}A_X \mathrm{MA}^{\mathrm{NA}}(U_\varepsilon)=\mathrm{Ent}^{\mathrm{NA}}(U_\varepsilon).
$$

On the other hand, using that $U^k\searrow U$ and $U_\varepsilon^k\searrow U_\varepsilon$ (recall the proof of Proposition \ref{prop: U_e_conti}), \cite[Theorem 6.9]{BoJ18a} ensures that
$$
\frac{1}{\varepsilon }\int_{X^{an}}(U_\varepsilon-U)\mathrm{MA}^{\mathrm{NA}}(U_\varepsilon)=\lim_{k\to\infty}\frac{1}{\varepsilon }\int_{X^{an}}(U^k_\varepsilon-U^k)\mathrm{MA}^{\mathrm{NA}}(U^k_\varepsilon)=\lim_{k\to\infty}\mathrm{Ent}^{\mathrm{NA}}(U^k_\varepsilon),
$$
where the last equality holds due to \eqref{eq: orth}. Using that $\mathrm{Ent}^{\mathrm{NA}}(\cdot)$ is lsc along the decreasing sequence $\{U^k_\varepsilon\}$, we obtain that
$$
\frac{1}{\varepsilon }\int_{X^{an}}(U_\varepsilon-U)\mathrm{MA}^{\mathrm{NA}}(U_\varepsilon)\geq \mathrm{Ent}^{\mathrm{NA}}(U_\varepsilon),
$$
thus finishing the proof.
\end{proof}

The following result can be extracted from the above proof.
\begin{corollary}
\label{cor: Ent=lim Ent}
    Let $\cH^{\mathrm{NA}}\ni U^k\searrow U$ be as in the previous proof. Then
    $$
    \mathrm{Ent}^{\mathrm{NA}}(U_\varepsilon)=\lim_{k\to\infty} \mathrm{Ent}^{\mathrm{NA}}(U^k_\varepsilon)\text{ for any }\varepsilon>0. 
    $$
\end{corollary}

The orthogonality identity in the next result should be compared with equation \eqref{eq: def_u_e}.
\begin{corollary}
    One has $\frac{1}{\varepsilon}(U_\varepsilon-U)=A_X$ a.e. with respect to $\mathrm{MA}^{\mathrm{NA}}(U_\varepsilon).$
\end{corollary}

\begin{proof}
    Proposition \ref{prop:Ent_U_e=int_Ue_U} shows that
    $$
    \mathrm{Ent}^{\mathrm{NA}}(U_\varepsilon)=\int_{X^{an}}\frac{1}{\varepsilon}(U_\varepsilon-U)\mathrm{MA}^{\mathrm{NA}}(U_\varepsilon)=\int_{X^{an}} A_X\,\mathrm{MA}^{\mathrm{NA}}(U_\varepsilon).
    $$
    Since $\frac{1}{\varepsilon}(U_\varepsilon-U)\leq A_X$ on $X^{an}$ (recall Lemma \ref{lem: Ue<U+eA}), we conclude.
\end{proof}

More importantly, we see that $U_\varepsilon$ satisfies the Aubin--Yau type equation in some variational sense.

\begin{corollary}
\label{cor:U_e_is_max}
    For any $U\in\cE^{1,\mathrm{NA}}$ and $\varepsilon>0$, the potential $U_\varepsilon$ is the unique maximizer $V$ for the Ding type functional
    $$
\sup_{V\in \mathrm{PSH}^{\mathrm{NA}}\cap C^0(X^{an})}\left(\mathrm{L}^{\mathrm{NA}}(\varepsilon^{-1}(U-V))+\varepsilon^{-1}(I^{\mathrm{NA}}(V)-I^{\mathrm{NA}}(U))\right)
$$
such that the normalization condition
$
\mathrm{L}^{\mathrm{NA}}(\varepsilon^{-1}(U-V))=0
$
holds.
\end{corollary}

Recall that $\mathrm{L}^{\mathrm{NA}}(f)$ for some function $f$ on $X^{an}$ is defined by
$$
\mathrm{L}^{\mathrm{NA}}(f):=\inf_{X^{div}}(A_X+f).
$$

\begin{proof}
    For any $V\in \mathrm{PSH}^{\mathrm{NA}}\cap C^0(X^{an})$, we first note that $
\mathrm{L}^{\mathrm{NA}}(\varepsilon^{-1}(U-V))>-\infty.
$
Indeed, this follows from the continuity of $V$ and \eqref{eq:U+A>_C}.
So after adding some constant to $V$ we can assume that $
\mathrm{L}^{\mathrm{NA}}(\varepsilon^{-1}(U-V))=0.
$
Then we have that
$$
A_X+\varepsilon^{-1}(U-V)\geq 0\text{ on }X^{div}.
$$
Since $A_X$ is the largest lsc extension of $A_X$ on $X^{div}$ and $V-U$ is lsc on $X^{an}$, the above inequality holds on $X^{an}$ as well.
Thus, $V\leq U+\varepsilon A_X$ on $X^{an}$, implying that $V\leq U_\varepsilon$. Consequently, $I^{\mathrm{NA}}(V)\leq I^{\mathrm{NA}}(U_\varepsilon)$. 

It remains to show that $
\mathrm{L}^{\mathrm{NA}}(\varepsilon^{-1}(U-U_\varepsilon))=0.
$
Namely, we need to show that
$$
\inf_{X^{div}}(A_X+\varepsilon^{-1}(U-U_\varepsilon))=0.
$$
Indeed, on the one hand we have (by Lemma \ref{lem: Ue<U+eA}) that
$$
\inf_{X^{div}}(A_X+\varepsilon^{-1}(U-U_\varepsilon))\geq 0.
$$
On the other hand, evaluating at $v_{triv}$, we deduce from \cite[(6.1)]{BBJ18} and \eqref{eq: sup_Ue=sup_U} that
$$
A_X(v_{triv})+\varepsilon^{-1}(U(v_{triv})-U_\varepsilon(v_{triv}))=\varepsilon^{-1}(\sup U-\sup U_\varepsilon)=0.
$$
Thus we conclude.
\end{proof}

We also have the following.

\begin{lemma}
\label{lem: Ent_Ue_conti_in_e}
    For any $U\in\cE^{1,\mathrm{NA}}$, the map $(0,\infty)\ni\varepsilon\mapsto \mathrm{Ent}^{\mathrm{NA}}(U_\varepsilon)$ is continuous.
\end{lemma}

\begin{proof}
    For any sequence $\varepsilon_i\to \varepsilon$, we have $U_{\varepsilon_i}\to U_\varepsilon$ uniformly
    due to Corollary \ref{cor: U_e_Lip_in_e}, so we deduce from \cite[Corollary 6.14]{BoJ18a} and Proposition \ref{prop:Ent_U_e=int_Ue_U} that
    \begin{equation*}
        \begin{split}
            \mathrm{Ent}^{\mathrm{NA}}(U_\varepsilon)&=\frac{1}{\varepsilon}\int_{X^{an}}(U_\varepsilon-U)\mathrm{MA}^{\mathrm{NA}}(U_\varepsilon)\\
            &=\lim_{i\to\infty}\frac{1}{\varepsilon_i}\int_{X^{an}}(U_{\varepsilon_i}-U)\mathrm{MA}^{\mathrm{NA}}(U_{\varepsilon_i})=\lim_{i\to\infty}\mathrm{Ent}^{\mathrm{NA}}(U_{\varepsilon_i}),\\
        \end{split}
    \end{equation*}
  finishing the proof.
\end{proof}

\begin{lemma}
\label{lem: lim_Ent_Ue=Ent_U}
    For any $U\in\cE^{1,\mathrm{NA}}$, one has $\lim_{\varepsilon\to 0^+}\mathrm{Ent}^{\mathrm{NA}}(U_\varepsilon)=\mathrm{Ent}^{\mathrm{NA}}(U).$
\end{lemma}

\begin{proof}
    For any $\varepsilon>0$, it holds that
    \begin{equation*}
        \begin{split}
             \mathrm{Ent}^{\mathrm{NA}}(U_\varepsilon)&=\frac{1}{\varepsilon}\int_{X^{an}}(U_\varepsilon-U)\mathrm{MA}^{\mathrm{NA}}(U_\varepsilon)\\
             &\leq \frac{1}{\varepsilon}(I^{\mathrm{NA}}(U_\varepsilon)-I^{\mathrm{NA}}(U))\leq \frac{1}{\varepsilon}\int_{X^{an}}(U_\varepsilon-U)\mathrm{MA}^{\mathrm{NA}}(U)\\
             &\leq \int_{X^{an}}A_X \mathrm{MA}^{\mathrm{NA}}(U)=\mathrm{Ent}^{\mathrm{NA}}(U),
        \end{split}
    \end{equation*}
    where we used Proposition \ref{prop:Ent_U_e=int_Ue_U} in the first line, \cite[(7.23)]{BJ22b} in the second line and Lemma \ref{lem: Ue<U+eA} in the last line.

    Since $U_\varepsilon\searrow U$ (Corollary \ref{cor:Ue_to_U}), we have
    $\mathrm{Ent}^{\mathrm{NA}}(U)\leq\liminf_{\varepsilon\to 0^+}\mathrm{Ent}^{\mathrm{NA}}(U_\varepsilon).$ Thus we conclude.
\end{proof}

\begin{remark}
       As we pointed out in Remark \ref{rem: cont_NApsh_induced_fromfilt}, each $U_\varepsilon$ is induced by some filtration of $R(X,L)$, so the entropy of any $U\in\cE^{1,\mathrm{NA}}$ can be well approximated using a sequence of filtrations, which in particular reproves a weaker version of Li's approximation result \cite[Proposition 6.3]{Li20-cscK}.
\end{remark}

In fact, combining Lemma \ref{lem: lim_Ent_Ue=Ent_U} with Corollary \ref{cor: Ent=lim Ent}, we immediately obtain the following more refined regularization result.

\begin{theorem}[=Theorem \ref{thm: main_reg_NA_ent}]
    For any $U\in\cE^{1,\mathrm{NA}}$, there exists a sequence $U^k\in\cH^{\mathrm{NA}}$ with $U^k\searrow U$ and a rational sequence $\varepsilon_k\searrow 0$ such that $U^k_{\varepsilon_k}\searrow U$ and
    $$
    \mathrm{Ent}^{\mathrm{NA}}(U)=\lim_{k\to\infty}\mathrm{Ent}^{\mathrm{NA}}(U^k_{\varepsilon_k}).
    $$
\end{theorem}

Finally, we prove the NA analogue of Theorem \ref{thm:main_f'}.
\begin{theorem}[=Theorem \ref{thm: main_F'}]
\label{thm: F'}
    For any $U\in\cE^{1,\mathrm{NA}}$, the function
   $F(\varepsilon):=I^{\mathrm{NA}}(U_\varepsilon)$ is a non-decreasing $C^1$ concave function on $[0,+\infty).$ More precisely,
   $$
   F'_+(0)=\mathrm{Ent}^{\mathrm{NA}}(U)\text{ and }F'(\varepsilon)=\mathrm{Ent}^{\mathrm{NA}}(U_\varepsilon)\text{ for any }\varepsilon>0.
   $$
\end{theorem}

\begin{proof}
    We can assume that $\sup U=0$, and pick $\cH^{\mathrm{NA}}\ni U^k\searrow U$ as in the proof of Proposition \ref{prop: U_e_conti}. Define
    $$
    F_k(\varepsilon):=I^{\mathrm{NA}}(U^k_\varepsilon)\text{ for }\varepsilon\geq 0.
    $$
    We will show that the result holds for each $F_k$ and then let $k\to\infty$.

    Each $U^k$ is determined by some smooth test configuration $(\cX^k,\cL^k)$ dominating $X\times\PP^1$. Then for any $\varepsilon\in\QQ_{\geq 0}$, $U^k_\varepsilon$ is determined by the log discrepancy model $(\cX^k,\cL^k+\varepsilon K^{\log}_{\cX^k/X\times\PP^1})$ (see \cite[Remark 9.2]{DZ25} and \eqref{eq: xia}). So from \cite[(27)]{Li21} we obtain that
    $$
    I^{\mathrm{NA}}(U^k_\varepsilon)=\frac{\vol(\cL^k+c\cX^k_0+\varepsilon K^{\log}_{\cX^k/X\times\PP^1})}{(n+1)V}-c,
    $$
    Fix $\varepsilon_0\geq0$. On a neighborhood of $\varepsilon_0$, choose one sufficiently large constant $c$ so that $\cL^k+c\cX^k_0+\varepsilon K^{\log}_{\cX^k/X\times\PP^1}$ is big throughout that neighborhood. Both sides of the above identity are continuous with respect to $\varepsilon$, so the identity extends from rational to real $\varepsilon$ there. The continuity of the left-hand side follows from Corollary \ref{cor: U_e_Lip_in_e}, while that of the right-hand side is standard (see, e.g., \cite{Laz-bookI}). Since $\varepsilon_0$ was arbitrary, \cite[Theorem A]{BFJ09} yields
    $
    F_k(\varepsilon)\in C^1([0,+\infty)),
    $
    and (for all $\varepsilon\geq 0$)
    \begin{equation*}
        \begin{split}
            F'_k(\varepsilon)&=\frac{1}{V}\langle(\cL^k+c\cX^k_0+\varepsilon K^{\log}_{\cX^k/X\times\PP^1})^n\rangle\cdot K^{\log}_{\cX^k/X\times\PP^1}\\
            &=\mathrm{Ent}^{\mathrm{NA}}(U_\varepsilon^k+c)=\mathrm{Ent}^{\mathrm{NA}}(U_\varepsilon^k),
        \end{split}
    \end{equation*}
    where we used \cite[(25)]{Li21} in the second line.
    This also implies that $F_k$ is non-decreasing, as $F'_k(\varepsilon)=\mathrm{Ent}^{\mathrm{NA}}(U^k_\varepsilon)\geq 0.$

    Now we show that $F_k$ is concave on $[0,+\infty)$. Picking any $0\leq \alpha<\delta<\gamma$, we want to show that
    $$
    \frac{F_k(\delta)-F_k(\alpha)}{\delta-\alpha}\geq\frac{F_k(\gamma)-F_k(\delta)}{\gamma-\delta}.
    $$
    Indeed, we have
    \begin{equation*}
        \begin{split}
            \frac{F_k(\delta)-F_k(\alpha)}{\delta-\alpha}&=\frac{1}{\delta-\alpha}(I^{\mathrm{NA}}(U^k_\delta)-I^{\mathrm{NA}}(U^k_\alpha))\\
            &\geq\frac{1}{(\delta-\alpha)}\int_{X^{an}}(U^k_\delta-U^k_\alpha)\mathrm{MA}^{\mathrm{NA}}(U^k_\delta)\\
            &=\frac{1}{(\delta-\alpha)}\int_{X^{an}}(U^k+\delta A_X\circ r_{\cX^k}-U^k_\alpha)\mathrm{MA}^{\mathrm{NA}}(U^k_\delta)\\
            &\geq \int_{X^{an}}A_X\circ r_{\cX^k}\mathrm{MA}^{\mathrm{NA}}(U^k_\delta)=\mathrm{Ent}^{\mathrm{NA}}(U^k_\delta).\\
        \end{split}
    \end{equation*}
    The second line is due to \cite[(7.23)]{BJ22b}. The third line used \eqref{eq: orth}. The last line follows from \eqref{eq: Uke<U+eAr} and \eqref{eq: orth}.
    For similar reasons, we also have
    \begin{equation*}
        \begin{split}
            \frac{F_k(\gamma)-F_k(\delta)}{\gamma-\delta}&=\frac{1}{\gamma-\delta}(I^{\mathrm{NA}}(U^k_\gamma)-I^{\mathrm{NA}}(U^k_\delta))\\
            &\leq\frac{1}{(\gamma-\delta)}\int_{X^{an}}(U^k_\gamma-U^k_\delta)\mathrm{MA}^{\mathrm{NA}}(U^k_\delta)\\
            &=\frac{1}{(\gamma-\delta)}\int_{X^{an}}(U^k_\gamma-U^k-\delta A_X\circ r_{\cX^k})\mathrm{MA}^{\mathrm{NA}}(U^k_\delta)\\
            &\leq \int_{X^{an}}A_X\circ r_{\cX^k}\mathrm{MA}^{\mathrm{NA}}(U^k_\delta)=\mathrm{Ent}^{\mathrm{NA}}(U^k_\delta).\\
        \end{split}
    \end{equation*}
    So we deduce that $F_k$ is concave on $[0,+\infty)$. Therefore, the result we aim to show holds for $F_k$.

    Since $U^k_\varepsilon\searrow U_\varepsilon$ for any $\varepsilon\geq 0$, by \cite[Proposition 6.2]{BoJ18a} we get that $F_k\searrow F$ pointwise as $k\nearrow\infty$. Thus $F$ is non-decreasing and concave on $[0,+\infty)$. Moreover, it is continuous at $\varepsilon=0$, by Corollary \ref{cor:Ue_to_U}. It remains to show that $F$ belongs to $C^1([0,+\infty))$.

For any $\varepsilon>0$ and $s\in(0,\varepsilon)$, we can write
\begin{equation*}
    \begin{split}
        F'_+(\varepsilon)&\geq \frac{F(\varepsilon+s)-F(\varepsilon)}{s}=\lim_{k\to\infty}\frac{F_k(\varepsilon+s)-F_k(\varepsilon)}{s}\\
        &\geq \lim_{k\to\infty}F'_k(\varepsilon+s)=\lim_{k\to\infty} \mathrm{Ent}^{\mathrm{NA}}(U^k_{\varepsilon+s})=\mathrm{Ent}^{\mathrm{NA}}(U_{\varepsilon+s}),\\
    \end{split}
\end{equation*}
where we used Corollary \ref{cor: Ent=lim Ent} in the last line.
Similarly,
\begin{equation*}
    \begin{split}
        F'_-(\varepsilon)&\leq \frac{F(\varepsilon)-F(\varepsilon-s)}{s}=\lim_{k\to\infty}\frac{F_k(\varepsilon)-F_k(\varepsilon-s)}{s}\\
        &\leq \lim_{k\to\infty}F'_k(\varepsilon-s)=\lim_{k\to\infty} \mathrm{Ent}^{\mathrm{NA}}(U^k_{\varepsilon-s})=\mathrm{Ent}^{\mathrm{NA}}(U_{\varepsilon-s}).\\
    \end{split}
\end{equation*}
Now using Lemma \ref{lem: Ent_Ue_conti_in_e}, we obtain that
$$
F'(\varepsilon)=\mathrm{Ent}^{\mathrm{NA}}(U_\varepsilon)\text{ for any }\varepsilon>0.
$$

Finally, we show that
$$
F'_+(0)=\mathrm{Ent}^{\mathrm{NA}}(U).
$$
By concavity, we have
$$
\frac{1}{\varepsilon }(F(\varepsilon)-F(0))\geq F'(\varepsilon)=\mathrm{Ent}^{\mathrm{NA}}(U_\varepsilon).
$$
Sending $\varepsilon\searrow 0$ and using Lemma \ref{lem: lim_Ent_Ue=Ent_U} we get
$
F'_+(0)\geq \mathrm{Ent}^{\mathrm{NA}}(U).
$
On the other hand, for any $\varepsilon>0$ we have (using Lemma \ref{lem: Ue<U+eA} again)
\begin{equation*}
    \begin{split}
        \frac{1}{\varepsilon }(F(\varepsilon)-F(0))&=\frac{1}{\varepsilon}(I^{\mathrm{NA}}(U_\varepsilon)-I^{\mathrm{NA}}(U))
        \leq\frac{1}{\varepsilon }\int_{X^{an}}(U_\varepsilon-U)\mathrm{MA}^{\mathrm{NA}}(U)\\
        &\leq \int_{X^{an}}A_X\mathrm{MA}^{\mathrm{NA}}(U)=\mathrm{Ent}^{\mathrm{NA}}(U).
    \end{split}
\end{equation*}
    So we obtain
    $
F'_+(0)\leq \mathrm{Ent}^{\mathrm{NA}}(U).
$
Thus $F'_+(0)=\mathrm{Ent}^{\mathrm{NA}}(U)$ holds, and we also have $F'(\varepsilon)\to F'_+(0)$ as $\varepsilon\to 0$ by Lemma \ref{lem: lim_Ent_Ue=Ent_U}.
This completes the proof.
\end{proof}

\begin{corollary}
    For any $U\in\cE^{1,\mathrm{NA}}$, one has $\mathrm{Ent}^{\mathrm{NA}}(U_\varepsilon)\nearrow \mathrm{Ent}^{\mathrm{NA}}(U)$ as $\varepsilon\searrow 0$.
\end{corollary}

\begin{proof}
    That $\mathrm{Ent}^{\mathrm{NA}}(U_\varepsilon)\to \mathrm{Ent}^{\mathrm{NA}}(U)$ is simply Lemma \ref{lem: lim_Ent_Ue=Ent_U}. The monotonicity follows from the concavity of $F$ and that $F'(\varepsilon)=\mathrm{Ent}^{\mathrm{NA}}(U_\varepsilon)$.
\end{proof}

The discussions of this section essentially only used some formal properties of NA geometry.
Thus, given the recent developments in \cite{Pic24,MPWN}, we believe it is possible to extend this section to the transcendental setting as well.

\section{Uniting the three `kingdoms'}

In this section we bring together the radial, non-Archimedean and potential theoretic viewpoints on interpreting entropy as a derivative of Aubin-Yau deformations.

Let $(X,\omega)$ be a projective manifold, with $\omega\in c_1(L)$ for some ample line bundle $L$ on $X$. 
In this part we fix a geodesic ray $\{u_t\}\in\cR^1$. In previous sections we have investigated three types of Aubin--Yau deformations, which we briefly recap.
For any $\varepsilon>0$ let $\{u_{t,\varepsilon}\}$ be the solutions of \eqref{eq: def_u_e} along $\{u_t\}$. According to \cite[\S 6]{DZ25} we have the following radial energy
$$
I\{u_{t,\varepsilon}\}:=\lim_{t\to\infty}\frac{I(u_{t,\varepsilon})}{t}.
$$

On the other hand, we also have that
$$
I\{P_\varepsilon(u)_t\}:=\lim_{t\to\infty}\frac{I(P_\varepsilon(u)_t)}{t},
$$
where $\{P_\varepsilon(u)_t\}$ is the radial Aubin--Yau deformation of $\{u_t\}$ introduced in \S \ref{sec: rad}.
Finally, on the non-Archimedean side, let $U\in\cE^{1,\mathrm{NA}}$ be the NA potential associated to the ray $\{u_t\}$ (in the sense of \cite[Theorem 6.4]{BBJ18}). Let $U_\varepsilon$ be the envelope defined in \eqref{eq: def_U_e}. We then have the NA Monge--Amp\`ere energy
$$
I^{\mathrm{NA}}(U_\varepsilon).
$$

The goal of this part is to prove the following slope identity that unifies the Aubin--Yau deformations in all three settings.

\begin{theorem}
\label{thm: 3kingdoms}
    For any $\{u_t\}\in\cR^1$ and $\varepsilon>0$ we have
    $$
    I\{u_{t,\varepsilon}\}=I\{P_\varepsilon(u)_t\}=I^{\mathrm{NA}}(U_\varepsilon).
    $$
\end{theorem}

That $I\{u_{t,\varepsilon}\}=I\{P_\varepsilon(u)_t\}$ holds follows from Theorem \ref{thm: intro_slope}. So it suffices to argue that
$$
I\{u_{t,\varepsilon}\}=I^{\mathrm{NA}}(U_\varepsilon).
$$

Let us first treat the case where $\{u_t\}$ a maximal geodesic ray (in the sense of \cite[\S 6.4]{BBJ18}). This means that there exists a decreasing sequence of geodesic rays $\{u^k_t\}$ arising from test configurations such that $\{u^k_t\}\rightarrow_{d_1^c}\{u_t\}$. 

\begin{proposition}
\label{prop: rad_E=NA_E}
    For any maximal geodesic ray $\{u_t\}$ and $\varepsilon\geq 0$, one has
    $$
   I\{u_{t,\varepsilon}\}=I^{\mathrm{NA}}(U_\varepsilon).
    $$
\end{proposition}

\begin{proof}
    We can assume that $\sup u_t=0$. Let $\{u^k_t\}$ be the decreasing sequence arising from test configurations with $\{u^k_t\}\rightarrow_{d_1^c}\{u_t\}$, which is constructed using multiplier ideal approximation as in the proof of Proposition \ref{prop:U_e_bd}. Let $U^k\in\cH^{\mathrm{NA}}$ be the NA potential associated to $\{u_t^k\}$, so we also have $U^k\searrow U$.
    For each $k$ we can assume that $U^k$ is determined by some smooth test configuration $(\cX^k,\cL^k)$ over $X\times\PP^1$. For any $\varepsilon\in\QQ_{\geq 0}$, it is known from \cite[(7.2)]{X23} and \cite[Remark 9.2]{DZ25} that $U^k_\varepsilon$ is the NA potential of the log discrepancy model $(\cX^k,\cL^k+\varepsilon K^{\log}_{\cX^k/X\times\PP^1})$. In fact \cite{DZ25} explicitly treats the case when $\varepsilon\in[0,1)\cap\QQ$, which is enough for applications anyway, but a slight modification of the argument therein shows that the same result holds for all $\varepsilon\in\QQ_{\geq 0}$.

    In particular, we deduce from the slope formula \cite[Proposition 9.4]{DZ25} that
    $$
    I\{u_{t,\varepsilon}^k\}=I^{\mathrm{NA}}(U^k_\varepsilon),\text{ for any }\varepsilon\in\QQ_{\geq 0}.
    $$

    We claim that the above identity extends to all $\varepsilon\geq 0$. For this, it suffices to observe that both sides are concave in $\varepsilon$, by Theorems \ref{thm: f'} and \ref{thm: F'}. So the claim holds.
Then (the proof of) \cite[Theorem 5.2]{DZ25} yields that 
    $$
    I\{u_{t,\varepsilon}\}=\lim_{k\to\infty}I\{u_{t,\varepsilon}^k\}=\lim_{k\to\infty}I^{\mathrm{NA}}(U^k_\varepsilon) \text{ for any }\varepsilon\geq 0.
    $$
Recall from the proof of Proposition \ref{prop: U_e_conti} that $U^k_\varepsilon\searrow U_\varepsilon$. Thus we have $\lim_{k\to\infty}I^{\mathrm{NA}}(U^k_\varepsilon)=I^{\mathrm{NA}}(U_\varepsilon)$, as the Monge--Amp\`ere energy converges along decreasing sequences, finishing the proof.
\end{proof}

Now we remove the maximality assumption for $\{u_t\}$. 

\begin{theorem}
\label{thm: rad_E=sup}
    For any geodesic ray $\{u_t\}\in\cR^1$ and any $\varepsilon>0$, we have
    \begin{equation}
        \label{eq: rad_E=sup}
        I\{u_{t,\varepsilon}\}=I^{\mathrm{NA}}(U_\varepsilon)=\sup_{\{v_t\}}\left(\varepsilon\cdot\mathrm{L}_\varepsilon(\{v_t\},\{u_t\})+I\{v_t\}\right),
    \end{equation}
    where $\{v_t\}$ runs through all geodesic rays arising from test configurations of $(X,L)$, and
    $$
    \mathrm{L}_\varepsilon(\{v_t\},\{u_t\}):=\liminf_{t\to\infty}\left(-\frac{1}{t}\log\int_Xe^{\frac{1}{\varepsilon}(v_t-u_t)}\frac{\omega^n}{V} \right).
    $$
\end{theorem}

\begin{proof}
    It follows from \cite[(28) and Proposition 3.6]{DZ25} that
    \begin{equation}
        \label{eq: rad_E>sup}
        I\{u_{t,\varepsilon}\}\geq \sup_{\{v_t\}}\left(\varepsilon\cdot\mathrm{L}_\varepsilon(\{v_t\},\{u_t\})+I\{v_t\}\right).
    \end{equation}
    In what follows, we argue that this is actually an equality.

    Recall that (by \cite[Theorem 6.4]{BBJ18}) $U\in\cE^{1,\mathrm{NA}}$ is the NA potential associated to $\{u_t\}$ and
    $
    I^{\mathrm{NA}}(U)\geq I\{u_t\}.
    $
    Let $\{w_t\}$ be the maximal geodesic ray attached to $U$, so that $w_t\geq u_t$ (see \cite[Theorem 6.6]{BBJ18}). Moreover, we have that (by \cite[Corollary 6.7]{BBJ18})
    $$
    I^{\mathrm{NA}}(U)=I\{w_t\}\geq I\{u_t\},
    $$
    where the second inequality can be strict. 
    
    Given $\varepsilon>0$, by \cite[Lemma 3.1]{DZ25} we have $w_{t,\varepsilon}\geq u_{t,\varepsilon}$. So Proposition \ref{prop: rad_E=NA_E} gives
    $$
     I^{\mathrm{NA}}(U_\varepsilon)=I\{w_{t,\varepsilon}\}\geq I\{u_{t,\varepsilon}\}.
    $$
    By Proposition \ref{prop: U_e_conti} and Dini's lemma, we can find a sequence $\cH^{\mathrm{NA}}\ni V^k\searrow U_\varepsilon$  converging uniformly, say
    $$
    U_\varepsilon\leq V^k\leq U_\varepsilon+\frac{1}{k},\ k\geq 1.
    $$
    Let $\{v_t^k\}$ be the geodesic rays associated with $V^k$. 
    Then by Lemma \ref{lem:remove_bd} below,
    \begin{equation*}
        \begin{split}
             \mathrm{L}_\varepsilon(\{v^k_t\},\{u_t\})&=\inf_{w\in X^{div}}\left(A_X(w)+\varepsilon^{-1}(U(w)-V^k(w))\right)\\
             &\geq \inf_{w\in X^{div}}\left(A_X(w)+\varepsilon^{-1}(U(w)-U_\varepsilon(w))\right)-\frac{1}{\varepsilon k}\geq -\frac{1}{\varepsilon k},
        \end{split}
    \end{equation*}
    where we used Lemma \ref{lem: Ue<U+eA} in the last inequality.
   Moreover, we have
    $$
    I\{v^k_t\}=I^{\mathrm{NA}}(V^k)\geq I^{\mathrm{NA}}(U_\varepsilon)\geq I\{u_{t,\varepsilon}\}.
    $$
    So we obtain that (recall \eqref{eq: rad_E>sup})
    $$
    I\{u_{t,\varepsilon}\}\geq\varepsilon\mathrm{L}_\varepsilon(\{v^k_t\},\{u_t\})+I\{v^k_t\}\geq I^{\mathrm{NA}}(U_\varepsilon)-1/k\geq I\{u_{t,\varepsilon}\}-1/k.
    $$
   Sending $k\to\infty$ we obtain the desired equality
   \begin{equation*}
       I\{u_{t,\varepsilon}\}=I^{\mathrm{NA}}(U_\varepsilon)=\sup_{\{v_t\}}\left(\varepsilon\cdot\mathrm{L}_\varepsilon(\{v_t\},\{u_t\})+I\{v_t\}\right).
   \end{equation*}
\end{proof}

\begin{remark}
    The rays $\{v_t^k\}$ in the above proof form a maximizing sequence for \eqref{eq: rad_E=sup}. If the admissible class is enlarged from test-configuration rays to all maximal finite energy rays, then the maximal geodesic ray $\{v_t\}$ associated with $U_\varepsilon$ is a maximizer. Indeed, $v_t\geq u_t$ gives $\mathrm{L}_\varepsilon(\{v_t\},\{u_t\})\leq0$, while $v_t\leq v_t^k$ gives $\mathrm{L}_\varepsilon(\{v_t\},\{u_t\})\geq-1/k$. Thus $\mathrm{L}_\varepsilon(\{v_t\},\{u_t\})=0$. By Theorem \ref{thm: maximizer_main}, this ray coincides with $\{P_\varepsilon(u)_t\}$.
\end{remark}

The next result is used in the proof of Theorem \ref{thm: rad_E=sup}.
\begin{lemma}
\label{lem:remove_bd}
    Let $\{u_t\}\in\cR^1$, let $U$ be its associated NA potential, and let $\varepsilon>0$. For any geodesic ray $\{v_t\}$ induced by some $V\in\cH^{\mathrm{NA}}$, we have
    $$
    \mathrm{L}_\varepsilon(\{v_t\},\{u_t\})=\inf_{w\in X^{div}}\left(A_X(w)+\varepsilon^{-1}(U(w)-V(w))\right).
    $$
\end{lemma}

\begin{proof}
    When $\{u_t\}$ is linearly bounded, this is nothing but \cite[Proposition 8.1]{DZ25}. So here we are claiming that the same assertion holds without the boundedness assumption. Indeed, in the proof of \cite[Proposition 8.1]{DZ25} the boundedness condition is only used at two places. First, it implies that $\mathrm{L}_\varepsilon(\{v_t\},\{u_t\})$ is a finite number; second, it implies that $e^{\frac{1}{\varepsilon}(v_{-\log|z|^2}-u_{-\log|z|^2})}$ is integrable away from $X\times\{0\}$. We claim that both hold true for possibly unbounded $\{u_t\}\subset\cE^1$ as well.

    Indeed, since $\{v_t\}$ is always bounded: $-Ct\leq v_t\leq Ct$, to show both properties we can simply assume that $v_t\equiv 0$. Then $\mathrm{L}_\varepsilon(\{v_t\},\{u_t\})$ being a finite number follows from (the proof of) \cite[Theorem 4.1]{DZ22} (see also \cite[Lemma 6.5]{DZ25} in the latest version). Moreover, $e^{\frac{1}{\varepsilon}(v_{-\log|z|^2}-u_{-\log|z|^2})}$ being integrable away from $X\times\{0\}$ follows from \cite{Ze01}.
\end{proof}

%Finally, combining the above with our main Theorems \ref{thm:main_f'}, \ref{thm:intro_radial} and \ref{thm: main_F'}, we immediately get the following result.

% This corollary reads a bit redundant, as it can be inferred from many different results in this paper and also other ones.

%\begin{corollary}
%For any geodesic ray $\{u_t\}$ with $\mathrm{Ent}\{u_t\}<\infty$, the function $[0,+\infty)\ni\varepsilon\mapsto I\{u_{t,\varepsilon}\}$ is $C^1$, non-decreasing and concave. Moreover,
%    $$
%    \frac{d}{d\varepsilon}\bigg|_{\varepsilon=0^+} I\{u_{t,\varepsilon}\}=\frac{d}{d\varepsilon}\bigg|_{\varepsilon=0^+} I\{P_\varepsilon(u)_t\}=\frac{d}{d\varepsilon}\bigg|_{\varepsilon=0^+} I^{\mathrm{NA}}(U_\varepsilon)=\mathrm{Ent}\{u_{t}\}=\mathrm{Ent}^{\mathrm{NA}}(U)
%    $$
%    and for any $\varepsilon>0$ one has that
%    $$
%    \frac{d}{d\varepsilon} I\{u_{t,\varepsilon}\}=\frac{d}{d\varepsilon}I\{P_\varepsilon(u)_t\}=\frac{d}{d\varepsilon}I^{\mathrm{NA}}(U_\varepsilon)=\mathrm{Ent}\{u_{t,\varepsilon}\}=\mathrm{Ent}\{P_\varepsilon(u)_t\}=\mathrm{Ent}^{\mathrm{NA}}(U_\varepsilon).
 %   $$
%\end{corollary}

\section{Applications to the cscK problem}
\label{sec: cscK}

Let $(X,L)$ be a projective manifold, with $L$ being an ample line bundle on $X$. Let $V:=L^n$ and $\bar S:=\frac{-n K_X\cdot L^{n-1}}{L^{n}}.$ In \cite{DZ25} we introduced a new notion of stability, called uniform $K^\beta$-stability, which characterizes the existence of constant scalar curvature K\"ahler (cscK) metrics in $c_1(L)$. Let us briefly recall its definition.

\begin{definition}
[\cite{DZ25}]
\label{def: K_beta_st}
Let $\beta>0$ be a constant.
    Given any smooth test configuration $(\cX,\cL)$ dominating $X\times\PP^1$, define $\cL_\beta:=\cL+\frac{1}{\beta}K^{\log}_{\cX/X\times\PP^1}$,
    $$
    \mathrm{K}^\beta(\cX,\cL):=\frac{\beta(\langle\cL_\beta^{n+1}\rangle-\cL^{n+1})}{(n+1)V}+\frac{\bar S\cL^{n+1}}{(n+1)V}+\frac{K_X\cdot\cL^n}{V},
    $$
    $$
    \mathrm{J}(\cX,\cL):=\frac{\cL\cdot L^n}{V}-\frac{\cL^{n+1}}{(n+1)V}.
    $$
    Here $\langle\cdot\rangle$ denotes the movable intersection product (see \cite[(62)]{DZ25}).
    We say
    \begin{enumerate}
        \item $(X,L)$ is $K^\beta$-semistable if $\mathrm{K}^\beta(\cX,\cL)\geq 0$ holds for all $(\cX,\cL)$.
        \item $(X,L)$ is uniformly $K^\beta$-stable if there exists some $\varepsilon>0$ such that $\mathrm{K}^\beta(\cX,\cL)\geq \varepsilon \mathrm{J}(\cX,\cL)$ for all $(\cX,\cL)$.
    \end{enumerate}
    
\end{definition}

The above stability notions are motivated by the $K^\beta$ energy introduced in \cite{DZ25}, which we now turn to describe. Fix a reference metric $\omega\in c_1(L)$. For any $\beta>0$ and $u\in\cE^1$ let $u_{1/\beta}$ be the solution of \eqref{eq: def_u_e} with $\varepsilon:=1/\beta$. Recall from \eqref{eq: beta_entropy_def} that the $\beta$-entropy is defined as
$
\mathrm{Ent}^\beta(u):=\beta(I(u_{1/\beta})-I(u)).
$ Also recall the $\cJ$ energy:
$$
\cJ(u):=\frac{1}{V}\int_X u\Ric(\omega)\wedge\sum_{i=0}^{n-1}\omega^i\wedge\omega^{n-1-i}_u-\bar S I(u).
$$
Now we can define $K$ and $K^\beta$ energies as
$$
\mathrm{K}(u):=\mathrm{Ent}(u)-\cJ(u)  \ \ \textup{ and } \ \ 
\mathrm{K}^\beta(u):=\mathrm{Ent}^\beta(u)-\cJ(u).
$$

\begin{definition}
    We say $K$ energy is proper if for some $\varepsilon>0$ and $C>0$ we have
    $$
    \mathrm{K}(u)\geq\varepsilon (\sup u-I(u))-C,\ u\in\cE^1.
    $$
     Similarly, we say $K^\beta$ energy is proper if for some $\varepsilon>0$ and $C>0$ we have
    $$
    \mathrm{K}^\beta(u)\geq\varepsilon (\sup u-I(u))-C,\ u\in\cE^1.
    $$
\end{definition}

As shown in \cite[\S3]{DZ25}, the $K^\beta$ energy behaves like the Ding energy, and also regularizes the $K$ energy. For any $\{u_t\}\in\cR^1$, set
\[
\mathrm{K}^\beta\{u_t\}:=\lim_{t\to\infty}\frac{\mathrm{K}^\beta(u_t)}{t}
=\mathrm{Ent}^\beta\{u_t\}-\cJ\{u_t\},
\]
where the existence of the limit follows from the corresponding radial slope formulas. If $\{u_t\}$ is attached to a test configuration $(\cX,\cL)$, then \cite[Theorem 1.2]{DZ25} gives
\begin{equation}
    \label{eq: DZ_slope_formula}
    \mathrm{K}^\beta\{u_t\}=\mathrm{K}^\beta(\cX,\cL).
\end{equation}
So the stability notions appearing in Definition \ref{def: K_beta_st} can be alternatively thought of as slope stability of $K^\beta$ along geodesic rays induced from test configurations.

Recall the following known characterization for cscK metrics.

\begin{theorem}
    {(\cite{CC2} and \cite[Theorem 1.1 and Theorem 4.4]{DZ25})}
    \label{thm: DZ_main}The following statements are equivalent.
    \begin{enumerate}
        \item There exists a unique cscK metric in $c_1(L)$.
        \item $K$ energy is proper.
        \item $K^\beta$ energy is proper for some $\beta>0$.
        \item $(X,L)$ is uniformly $K^\beta$-stable for some $\beta>0$.
    \end{enumerate}
\end{theorem}

Now we present the main theorem of this section, which further sharpens the above result.

\begin{theorem}[$\Rightarrow$ Theorem \ref{thm:cscK}]
\label{thm:cscK'}
The following statements are equivalent.
    \begin{enumerate}
        \item There exists a unique cscK metric in $c_1(L)$.
        \item $K$ energy is proper.
        \item $K^\beta$ energy is proper for some $\beta>0$.
        \item $K^\beta$ energy is bounded from below for some $\beta>0$, i.e.,
        $$
        \inf_{u\in\cE^1}\mathrm{K}^\beta(u)>-\infty.
        $$
        \item  $(X,L)$ is uniformly $K^\beta$-stable for some $\beta>0$.
        \item $(X,L)$ is $K^\beta$-semistable for some $\beta>0$.
        \item For any geodesic ray $\{u_t\}\in\cR^1$ there exists $\beta>0$ such that $\mathrm{K}^\beta\{u_t\}\geq0$.
    \end{enumerate}
\end{theorem}

\begin{proof}
    Given Theorem \ref{thm: DZ_main}, Definition \ref{def: K_beta_st} and \eqref{eq: DZ_slope_formula}, it is enough to show that (6)$\Rightarrow$(1) and (1)$\Leftrightarrow$(7). 

    Let us first argue (6)$\Rightarrow$(1).
    By \cite[Theorem 7.1]{CC2} and \cite[Theorem 1.3]{BerBern17}, to show that there exists a unique cscK metric in $c_1(L)$, it amounts to proving that \[\mathrm{K}\{u_t\}:=\lim_{t\to\infty}\frac{\mathrm{K}(u_t)}{t}>0\] holds for any nontrivial geodesic ray $\{u_t\}\subset \cE^1$. Here by nontrivial we mean that $$
    \mathrm{J}\{u_t\}:=\lim_{t\to\infty}\left(\frac{\sup u_t}{t}-\frac{I(u_t)}{t}\right)>0.
    $$
    Equivalently, $u_t\neq ct$ for any $c\in\RR$.

    To this end, first note that the estimates in \cite[Proposition 4.11]{BBEGZ19} (see also \cite[\S3]{Zha20-conti}) imply the existence of $\delta>0$ and $C>0$ such that
    $$
    \mathrm{Ent}(u)\geq\delta (\sup u-I(u))-C,\ u\in\cE^1.
    $$
    So for any nontrivial geodesic ray $\{u_t\}$ we have that
    \begin{equation}
        \label{eq: rad_Ent>0}
        \mathrm{Ent}\{u_t\}:=\lim_{t\to\infty}\frac{\mathrm{Ent}(u_t)}{t}\geq\delta \mathrm{J}\{u_t\}>0.
    \end{equation}

In addition, we might as well assume that $\mathrm{K}\{u_t\}<\infty$ (otherwise $\mathrm{K}\{u_t\}>0$ trivially holds). Then by \cite[Theorem 1.2]{Li20-cscK} we can find a sequence of geodesic rays $\{u_t^i\}$ arising from test configurations such that $\{u_t^i\}\rightarrow_{d_1^c}\{u_t\}$. By \eqref{eq: DZ_slope_formula} and the assumption that $(X,L)$ is $K^\beta$-semistable for some $\beta>0$, we obtain that
$$
\mathrm{K}^\beta\{u^i_t\}\geq0\text{ for all }i.
$$
Sending $i\to\infty$ and using \cite[Theorem 5.2]{DZ25}, we deduce that
$$
\mathrm{K}^\beta\{u_t\}\geq 0.
$$

Note that our goal is to argue that
$$
\mathrm{K}\{u_t\}>0.
$$
Comparing $K$ and $K^\beta$ energies, it suffices to show that
\begin{equation}
    \label{eq: Ent>Ent_beta}
    \mathrm{Ent}\{u_t\}>\mathrm{Ent}^\beta\{u_t\}=\beta(I\{u_{t,1/\beta}\}-I\{u_t\}).
\end{equation}

To achieve this, one can either use the radial analysis in \S\ref{sec: rad} or the NA theory developed in \S\ref{sec: NA}. We will present a proof using the former. For $\varepsilon\in(0,1/\beta]$ put
$$
h(\varepsilon):=\frac{1}{\varepsilon}(I\{u_{t,\varepsilon}\}-I\{u_t\}).
$$
By Theorems \ref{thm:intro_radial} and \ref{thm: intro_slope}, $\varepsilon\mapsto h(\varepsilon)$ is $C^1$, non-increasing and $\lim_{\varepsilon\to 0^+}h(\varepsilon)=\mathrm{Ent}\{u_t\}.$ If it happens that $\mathrm{Ent}\{u_t\}=\mathrm{Ent}^\beta\{u_t\}$, then $\varepsilon\mapsto h(\varepsilon)$ is forced to be constant, so that $\varepsilon\mapsto I\{u_{t,\varepsilon}\}$ is affine. Applying Theorems \ref{thm:intro_radial} and \ref{thm: intro_slope} once again we deduce that
$$
\mathrm{Ent}\{u_t\}=h(\varepsilon)=\frac{d}{d\varepsilon}I\{u_{t,\varepsilon}\}=\mathrm{Ent}\{P(u)_t^{1/\varepsilon}\}\text{ for any }\varepsilon\in(0,1/\beta].
$$
Then Corollary \ref{cor: Ent_id} gives
$
\{u_t\}=\{P(u)_t^{1/\varepsilon}\}\text{ for any }\varepsilon\in(0,1/\beta],
$
which implies (by Theorem \ref{thm:intro_radial}) that
$$
\mathrm{Ent}\{u_t\}=\lim_{\varepsilon\to 0^+}\frac{d}{d\varepsilon}I\{P(u)_t^{1/\varepsilon}\}=0,
$$
contradicting our assumption \eqref{eq: rad_Ent>0}. So we obtain that (6)$\Rightarrow$(1).

Finally, we argue (1)$\Leftrightarrow$(7). In fact, $(1)$ implies $(4)$, which clearly implies (7). Conversely, for a nontrivial ray with $\mathrm{K}\{u_t\}=+\infty$, positivity is immediate; when $\mathrm{K}\{u_t\}<\infty$, the proof of (6)$\Rightarrow$(1) gives $\mathrm{K}\{u_t\}>\mathrm{K}^\beta\{u_t\}$ for every $\beta>0$. Hence (7) implies that $\mathrm{K}\{u_t\}>0$ for every nontrivial $\{u_t\}\in\cR^1$. We conclude by \cite[Theorem 7.1]{CC2} and \cite[Theorem 1.3]{BerBern17}.
\end{proof}

\let\oldthebibliography\thebibliography
\let\endoldthebibliography\endthebibliography
\renewenvironment{thebibliography}[1]{
  \begin{oldthebibliography}{#1}
    \setlength{\itemsep}{0pt}
    \setlength{\parskip}{0pt}
    \setlength{\parsep}{0pt}
}{\end{oldthebibliography}}

\small{
\bibliography{ref.bib}

@article {BFJ15a,
    AUTHOR = {Boucksom, S\'{e}bastien and Favre, Charles and Jonsson, Mattias},
     TITLE = {Solution to a non-{A}rchimedean {M}onge-{A}mp\`ere equation},
   JOURNAL = {J. Amer. Math. Soc.},
  FJOURNAL = {Journal of the American Mathematical Society},
    VOLUME = {28},
      YEAR = {2015},
    NUMBER = {3},
     PAGES = {617--667}
}

@misc{DZ25,
      title={A {YTD} correspondence for constant scalar curvature
metrics}, 
      author={Darvas, Tamas and Zhang, Kewei},
      year={2025},
      note={arXiv:2509.15173},
      archivePrefix={arXiv},
      primaryClass={math.DG}
}

@Article{GLZ18,
 Author = {Guedj, Vincent and Lu, Chinh H. and Zeriahi, Ahmed},
 Title = {Stability of solutions to complex {Monge}-{Amp{\`e}re} flows},
 FJournal = {Annales de l'Institut Fourier},
 Journal = {Ann. Inst. Fourier},
 ISSN = {0373-0956},
 Volume = {68},
 Number = {7},
 Pages = {2819--2836},
 Year = {2018}
}

@article {GZ07,
    AUTHOR = {Guedj, Vincent and Zeriahi, Ahmed},
     TITLE = {The weighted {M}onge-{A}mp\`ere energy of quasiplurisubharmonic
              functions},
   JOURNAL = {J. Funct. Anal.},
  FJOURNAL = {Journal of Functional Analysis},
    VOLUME = {250},
      YEAR = {2007},
    NUMBER = {2},
     PAGES = {442--482}
}

@article {DZ22,
    AUTHOR = {Darvas, Tam\'as and Zhang, Kewei},
     TITLE = {Twisted {K}\"ahler-{E}instein metrics in big classes},
   JOURNAL = {Comm. Pure Appl. Math.},
  FJOURNAL = {Communications on Pure and Applied Mathematics},
    VOLUME = {77},
      YEAR = {2024},
    NUMBER = {12},
     PAGES = {4289--4327},
      ISSN = {0010-3640,1097-0312},
   MRCLASS = {32Q20 (53C55)},
  MRNUMBER = {4814920},
       DOI = {10.1002/cpa.22206},
       URL = {https://doi.org/10.1002/cpa.22206},
}

@misc{BJ25preprint,
      title={{On the Yau-Tian-Donaldson conjecture for weighted cscK metrics}}, 
      author={Sebastien Boucksom and Mattias Jonsson},
      year={2025},
      note={arXiv:2509.15016},
      archivePrefix={arXiv},
      primaryClass={math.DG}
}

@article {BerBern17,
    AUTHOR = {Berman, Robert J. and Berndtsson, Bo},
     TITLE = {Convexity of the {$K$}-energy on the space of {K}\"{a}hler metrics
              and uniqueness of extremal metrics},
   JOURNAL = {J. Amer. Math. Soc.},
  FJOURNAL = {Journal of the American Mathematical Society},
    VOLUME = {30},
      YEAR = {2017},
    NUMBER = {4},
     PAGES = {1165--1196},
      ISSN = {0894-0347},
   MRCLASS = {32Q15 (53C55)},
  MRNUMBER = {3671939},
MRREVIEWER = {Ruadha\'{\i} Dervan},
       DOI = {10.1090/jams/880},
       URL = {https://doi.org/10.1090/jams/880},
}

@article {Li20-cscK,
    AUTHOR = {Li, Chi},
     TITLE = {Geodesic rays and stability in the csc{K} problem},
   JOURNAL = {Ann. Sci. \'{E}c. Norm. Sup\'{e}r. (4)},
  FJOURNAL = {Annales Scientifiques de l'\'{E}cole Normale Sup\'{e}rieure.
              Quatri\`eme S\'{e}rie},
    VOLUME = {55},
      YEAR = {2022},
    NUMBER = {6},
     PAGES = {1529--1574},
      ISSN = {0012-9593,1873-2151},
   MRCLASS = {32Q26 (14E05 32Q15)},
  MRNUMBER = {4517682},
MRREVIEWER = {Ahmed\ Lesfari},
}

@incollection {Li21,
    AUTHOR = {Li, Chi},
     TITLE = {K-stability and {F}ujita approximation},
 BOOKTITLE = {Birational geometry, {K}\"ahler-{E}instein metrics and
              degenerations},
    SERIES = {Springer Proc. Math. Stat.},
    VOLUME = {409},
     PAGES = {545--566},
 PUBLISHER = {Springer, Cham},
      YEAR = {2023},
      ISBN = {978-3-031-17858-0; 978-3-031-17859-7},
   MRCLASS = {32Q15 (14E30 32P05 32Q26)},
  MRNUMBER = {4606653},
MRREVIEWER = {R\'emi\ Reboulet},
       DOI = {10.1007/978-3-031-17859-7\_26},
       URL = {https://doi.org/10.1007/978-3-031-17859-7_26},
}

@article {Dar15,
    AUTHOR = {Darvas, Tam\'{a}s},
     TITLE = {The {M}abuchi geometry of finite energy classes},
   JOURNAL = {Adv. Math.},
  FJOURNAL = {Advances in Mathematics},
    VOLUME = {285},
      YEAR = {2015},
     PAGES = {182--219}
}

@article {CC2,
    AUTHOR = {Chen, Xiuxiong and Cheng, Jingrui},
     TITLE = {On the constant scalar curvature {K}\"{a}hler metrics
              ({II})---{E}xistence results},
   JOURNAL = {J. Amer. Math. Soc.},
  FJOURNAL = {Journal of the American Mathematical Society},
    VOLUME = {34},
      YEAR = {2021},
    NUMBER = {4},
     PAGES = {937--1009},
}

@article {SD18,
    AUTHOR = {Sj\"ostr\"om Dyrefelt, Zakarias},
     TITLE = {K-semistability of csc{K} manifolds with transcendental
              cohomology class},
   JOURNAL = {J. Geom. Anal.},
  FJOURNAL = {Journal of Geometric Analysis},
    VOLUME = {28},
      YEAR = {2018},
    NUMBER = {4},
     PAGES = {2927--2960},
      ISSN = {1050-6926,1559-002X},
   MRCLASS = {32Q26 (14D06 14F43 32Q15 53C56)},
  MRNUMBER = {3881961},
MRREVIEWER = {Jian\ Xiao},
       DOI = {10.1007/s12220-017-9942-9},
       URL = {https://doi.org/10.1007/s12220-017-9942-9},
}

@article {BFJ09,
    AUTHOR = {Boucksom, S\'{e}bastien and Favre, Charles and Jonsson, Mattias},
     TITLE = {Differentiability of volumes of divisors and a problem of
              {T}eissier},
   JOURNAL = {J. Algebraic Geom.},
  FJOURNAL = {Journal of Algebraic Geometry},
    VOLUME = {18},
      YEAR = {2009},
    NUMBER = {2},
     PAGES = {279--308},
      ISSN = {1056-3911}
}

@article {DL19,
    AUTHOR = {Darvas, Tam\'{a}s and Lu, Chinh H.},
     TITLE = {Geodesic stability, the space of rays and uniform convexity in
              {M}abuchi geometry},
   JOURNAL = {Geom. Topol.},
  FJOURNAL = {Geometry \& Topology},
    VOLUME = {24},
      YEAR = {2020},
    NUMBER = {4},
     PAGES = {1907--1967},
      ISSN = {1465-3060},
   MRCLASS = {32Q26 (32U05 53C55)},
  MRNUMBER = {4173924},
       DOI = {10.2140/gt.2020.24.1907},
       URL = {https://doi.org/10.2140/gt.2020.24.1907},
}

@misc{BoJ18a,
      title={Singular semipositive metrics on line bundles on varieties over trivially valued fields}, 
      author={SÃ©bastien Boucksom and Mattias Jonsson},
      year={2018},
      note={arXiv:1801.08229},
      primaryClass={math.AG}
}

@misc{LZ26,
      title={Sharp parameter gaps for quantized Mabuchi energies and cscK metrics}, 
      author={Kefeng Liu and Langdi Zhang},
      year={2026},
      note={arXiv:2610.16697},
      archivePrefix={arXiv},
      primaryClass={math.DG}
}

@misc{Pap26,
      title={An algebraic comparison of K-stability and $K^\beta$ stability}, 
      author={T. S. Papazachariou},
      year={2026},
      note={arXiv:2609.15333v1},
      archivePrefix={arXiv},
      primaryClass={math.DG}
}

@incollection {DarvasSurvey,
    AUTHOR = {Darvas, Tam\'{a}s},
     TITLE = {Geometric pluripotential theory on {K}\"{a}hler manifolds},
 BOOKTITLE = {Advances in complex geometry},
    SERIES = {Contemp. Math.},
    VOLUME = {735},
     PAGES = {1--104},
 PUBLISHER = {Amer. Math. Soc.},
      YEAR = {2019}
}

@article {CTW18,
    AUTHOR = {Chu, Jianchun and Tosatti, Valentino and Weinkove, Ben},
     TITLE = {{$C^{1,1}$} regularity for degenerate complex {M}onge-{A}mp\`ere
              equations and geodesic rays},
   JOURNAL = {Comm. Partial Differential Equations},
  FJOURNAL = {Communications in Partial Differential Equations},
    VOLUME = {43},
      YEAR = {2018},
    NUMBER = {2},
     PAGES = {292--312},
      ISSN = {0360-5302},
   MRCLASS = {32Q15 (32W20 35J96 53C55)},
  MRNUMBER = {3777876},
MRREVIEWER = {Vamsi Pritham Pingali},
       DOI = {10.1080/03605302.2018.1446167},
       URL = {https://doi.org/10.1080/03605302.2018.1446167},
}

@article {RWN14,
    AUTHOR = {Ross, Julius and Witt Nystr\"{o}m, David},
     TITLE = {Analytic test configurations and geodesic rays},
   JOURNAL = {J. Symplectic Geom.},
  FJOURNAL = {The Journal of Symplectic Geometry},
    VOLUME = {12},
      YEAR = {2014},
    NUMBER = {1},
     PAGES = {125--169},
      ISSN = {1527-5256},
   MRCLASS = {32L05 (32U05)},
  MRNUMBER = {3194078},
MRREVIEWER = {Bianca Santoro},
       DOI = {10.4310/JSG.2014.v12.n1.a5},
       URL = {https://doi.org/10.4310/JSG.2014.v12.n1.a5},
}

@book {Xubook,
    AUTHOR = {Xu, Chenyang},
     TITLE = {K-stability of {F}ano varieties},
    SERIES = {New Mathematical Monographs},
    VOLUME = {50},
 PUBLISHER = {Cambridge University Press, Cambridge},
      YEAR = {2025},
     PAGES = {xi+411},
      ISBN = {978-1-009-53877-0; [9781009538763]},
   MRCLASS = {14J45 (32Q26)},
  MRNUMBER = {4893062},
}

@article {BBEGZ19,
    AUTHOR = {Berman, Robert J. and Boucksom, Sebastien and Eyssidieux,
              Philippe and Guedj, Vincent and Zeriahi, Ahmed},
     TITLE = {K\"{a}hler-{E}instein metrics and the {K}\"{a}hler-{R}icci flow on log
              {F}ano varieties},
   JOURNAL = {J. Reine Angew. Math.},
  FJOURNAL = {Journal f\"{u}r die Reine und Angewandte Mathematik. [Crelle's
              Journal]},
    VOLUME = {751},
      YEAR = {2019},
     PAGES = {27--89},
      ISSN = {0075-4102},
   MRCLASS = {53C55 (53C25 53C44)},
  MRNUMBER = {3956691},
MRREVIEWER = {Kai Zheng},
       DOI = {10.1515/crelle-2016-0033},
       URL = {https://doi.org/10.1515/crelle-2016-0033},
}

@article {BBGZ,
    AUTHOR = {Berman, Robert J. and Boucksom, S\'{e}bastien and Guedj, Vincent
              and Zeriahi, Ahmed},
     TITLE = {A variational approach to complex {M}onge-{A}mp\`ere equations},
   JOURNAL = {Publ. Math. Inst. Hautes \'{E}tudes Sci.},
  FJOURNAL = {Publications Math\'{e}matiques. Institut de Hautes \'{E}tudes
              Scientifiques},
    VOLUME = {117},
      YEAR = {2013},
     PAGES = {179--245},
      ISSN = {0073-8301},
   MRCLASS = {32W20 (32Q15 32U05 35J93 53C55)},
  MRNUMBER = {3090260},
MRREVIEWER = {Rafa\l  Czy\.{z}},
       DOI = {10.1007/s10240-012-0046-6},
       URL = {https://doi.org/10.1007/s10240-012-0046-6},
}

@article {Ze01,
    AUTHOR = {Zeriahi, Ahmed},
     TITLE = {Volume and capacity of sublevel sets of a {L}elong class of
              plurisubharmonic functions},
   JOURNAL = {Indiana Univ. Math. J.},
  FJOURNAL = {Indiana University Mathematics Journal},
    VOLUME = {50},
      YEAR = {2001},
    NUMBER = {1},
}

@phdthesis{Sz06,
  author       = {Székelyhidi, Gábor},
  title        = {Extremal metrics and K-stability},
  school       = {Imperial College London},
  year         = {2006},
  type         = {PhD Thesis},
  url          = {https://arxiv.org/abs/math/0611002},
  eprint       = {math/0611002},
  eprinttype   = {arXiv},
}

@article {BHJ17,
    AUTHOR = {Boucksom, S\'{e}bastien and Hisamoto, Tomoyuki and Jonsson,
              Mattias},
     TITLE = {Uniform {K}-stability, {D}uistermaat-{H}eckman measures and
              singularities of pairs},
   JOURNAL = {Ann. Inst. Fourier (Grenoble)},
  FJOURNAL = {Universit\'{e} de Grenoble. Annales de l'Institut Fourier},
    VOLUME = {67},
      YEAR = {2017},
    NUMBER = {2},
     PAGES = {743--841},
      ISSN = {0373-0956},
   MRCLASS = {32Q26 (14E30 14G22)},
  MRNUMBER = {3669511},
MRREVIEWER = {Yuji Odaka},
       URL = {http://aif.cedram.org/item?id=AIF_2017__67_2_743_0},
}

@article {Der14,
    AUTHOR = {Dervan, Ruadha\'i},
     TITLE = {Uniform stability of twisted constant scalar curvature
              {K}\"ahler metrics},
   JOURNAL = {Int. Math. Res. Not. IMRN},     
    YEAR = {2016},
    NUMBER = {15},
     PAGES = {4728--4783},
      ISSN = {1073-7928,1687-0247},
   MRCLASS = {32Q26 (32Q15 53C55)},
  MRNUMBER = {3564626},
MRREVIEWER = {Eleonora\ Di Nezza},
       DOI = {10.1093/imrn/rnv291},
       URL = {https://doi.org/10.1093/imrn/rnv291},
}

@misc{T26,
      title={A solution to the {Yau-Tian-Donaldson conjecture through special F}ujita Approximations}, 
      author={Antonio Trusiani},
      year={2026},
      note={arXiv:2605.30063},
      archivePrefix={arXiv},
      primaryClass={math.AG},
      url={https://arxiv.org/abs/2605.30063}, 
}

@inbook{Oh14, place={Cambridge}, series={London Mathematical Society Lecture Note Series}, title={Ricci curvature, entropy, and optimal transport}, booktitle={Optimal Transport: Theory and Applications}, publisher={Cambridge University Press}, author={Ohta, Shin-Ichi}, editor={Ollivier, Yann and Pajot, Hervé and Villani, CedricEditors}, year={2014}, pages={145–200}, collection={London Mathematical Society Lecture Note Series}}

@article{GT25,
 author = {Guedj, Vincent and Tat Dat T{\^o}},
 title = {K{\"a}hler families of {Green}'s functions},
 fjournal = {Journal de l'{\'E}cole Polytechnique -- Math{\'e}matiques},
 journal = {J. {\'E}c. Polytech., Math.},
 issn = {2429-7100},
 volume = {12},
 pages = {319--339},
 year = {2025},
 language = {English},
 doi = {10.5802/jep.291},
 zbMATH = {7991840},
 Zbl = {1571.32064}
}

@article{GPSS24,
 author = {Guo, Bin and Phong, Duong H. and Song, Jian and Sturm, Jacob},
 title = {Diameter estimates in {K{\"a}hler} geometry},
 fjournal = {Communications on Pure and Applied Mathematics},
 journal = {Commun. Pure Appl. Math.},
 issn = {0010-3640},
 volume = {77},
 number = {8},
 pages = {3520--3556},
 year = {2024},
 language = {English},
 doi = {10.1002/cpa.22196},
 zbMATH = {7896925},
 Zbl = {1547.53083}
}

@misc{Pe02,
      title={The entropy formula for the Ricci flow and its geometric applications}, 
      author={Grisha Perelman},
      year={2002},
      eprint={math/0211159},
      archivePrefix={arXiv},
      primaryClass={math.DG},
      url={https://arxiv.org/abs/math/0211159}, 
}

@article {GZ15,
    AUTHOR = {Guan, Qi'an and Zhou, Xiangyu},
     TITLE = {A proof of {D}emailly's strong openness conjecture},
   JOURNAL = {Ann. of Math. (2)},
  FJOURNAL = {Annals of Mathematics. Second Series},
    VOLUME = {182},
      YEAR = {2015},
    NUMBER = {2},
     PAGES = {605--616},
      ISSN = {0003-486X},
   MRCLASS = {32U05},
  MRNUMBER = {3418526},
MRREVIEWER = {\.{Z}ywomir Dinew},
       DOI = {10.4007/annals.2015.182.2.5},
       URL = {https://doi.org/10.4007/annals.2015.182.2.5},
}

@article {BDL17,
    AUTHOR = {Berman, Robert J. and Darvas, Tam\'as and Lu, Chinh H.},
     TITLE = {Convexity of the extended {K}-energy and the large time
              behavior of the weak {C}alabi flow},
   JOURNAL = {Geom. Topol.},
  FJOURNAL = {Geometry \& Topology},
    VOLUME = {21},
      YEAR = {2017},
    NUMBER = {5},
     PAGES = {2945--2988},
      ISSN = {1465-3060,1364-0380},
   MRCLASS = {53C55 (32U05 32W20)},
  MRNUMBER = {3687111},
}

@article{X23,
 author = {Xia, Mingchen},
 title = {Pluripotential-theoretic stability thresholds},
 fjournal = {IMRN. International Mathematics Research Notices},
 journal = {Int. Math. Res. Not.},
 issn = {1073-7928},
 volume = {2023},
 number = {14},
 pages = {12324--12382},
 year = {2023},
 language = {English},
 doi = {10.1093/imrn/rnac186},
 zbMATH = {7726867},
 Zbl = {1523.32030}
}

@article {Tian97,
    AUTHOR = {Tian, Gang},
     TITLE = {K\"{a}hler-{E}instein metrics with positive scalar curvature},
   JOURNAL = {Invent. Math.},
  FJOURNAL = {Inventiones Mathematicae},
    VOLUME = {130},
      YEAR = {1997},
    NUMBER = {1},
     PAGES = {1--37},
      ISSN = {0020-9910},
   MRCLASS = {53C25 (32L07 53C55)},
  MRNUMBER = {1471884},
MRREVIEWER = {Thalia D. Jeffres},
       DOI = {10.1007/s002220050176},
       URL = {https://doi.org/10.1007/s002220050176},
}

@article {Don02,
    AUTHOR = {Donaldson, S. K.},
     TITLE = {Scalar curvature and stability of toric varieties},
   JOURNAL = {J. Differential Geom.},
  FJOURNAL = {Journal of Differential Geometry},
    VOLUME = {62},
      YEAR = {2002},
    NUMBER = {2},
     PAGES = {289--349},
      ISSN = {0022-040X},
   MRCLASS = {32Q15 (14M25 53C21 53C55)},
  MRNUMBER = {1988506},
       URL = {http://projecteuclid.org/euclid.jdg/1090950195},
}

@article {PS10,
    AUTHOR = {Phong, D. H. and Sturm, Jacob},
     TITLE = {The {D}irichlet problem for degenerate complex
              {M}onge-{A}mp\`{e}re equations},
   JOURNAL = {Comm. Anal. Geom.},
  FJOURNAL = {Communications in Analysis and Geometry},
    VOLUME = {18},
      YEAR = {2010},
    NUMBER = {1},
     PAGES = {145--170},
      ISSN = {1019-8385},
   MRCLASS = {32W20 (35J96 35R01)},
  MRNUMBER = {2660461},
MRREVIEWER = {Yanir A. Rubinstein},
       DOI = {10.4310/CAG.2010.v18.n1.a6},
       URL = {https://doi.org/10.4310/CAG.2010.v18.n1.a6},
}

@article {BBN,
    AUTHOR = {Berman, Robert and Boucksom, S\'{e}bastien and Witt
              Nystr\"{o}m, David},
     TITLE = {Fekete points and convergence towards equilibrium measures on
              complex manifolds},
   JOURNAL = {Acta Math.},
  FJOURNAL = {Acta Mathematica},
    VOLUME = {207},
      YEAR = {2011},
    NUMBER = {1},
     PAGES = {1--27}
}

@article {PS07a,
    AUTHOR = {Phong, Duong H. and Sturm, Jacob},
     TITLE = {Test configurations for {K}-stability and geodesic rays},
   JOURNAL = {J. Symplectic Geom.},
  FJOURNAL = {The Journal of Symplectic Geometry},
    VOLUME = {5},
      YEAR = {2007},
    NUMBER = {2},
     PAGES = {221--247},
      ISSN = {1527-5256},
   MRCLASS = {32Q15 (32L10 32W20 58E11)},
  MRNUMBER = {2377252},
MRREVIEWER = {Julien Keller},
       DOI = {10.4310/jsg.2007.v5.n2.a3},
       URL = {https://doi.org/10.4310/jsg.2007.v5.n2.a3},
}

@book {Laz-bookI,
    AUTHOR = {Lazarsfeld, Robert},
     TITLE = {Positivity in algebraic geometry. {II}},
    SERIES = {Ergebnisse der Mathematik und ihrer Grenzgebiete. 3. Folge. A
              Series of Modern Surveys in Mathematics [Results in
              Mathematics and Related Areas. 3rd Series. A Series of Modern
              Surveys in Mathematics]},
    VOLUME = {49},
      NOTE = {Positivity for vector bundles, and multiplier ideals},
 PUBLISHER = {Springer-Verlag, Berlin},
      YEAR = {2004},
     PAGES = {xviii+385},
      ISBN = {3-540-22534-X},
   MRCLASS = {14-02 (14C20 14F05 14F17)},
  MRNUMBER = {2095472},
MRREVIEWER = {Mihnea Popa},
       DOI = {10.1007/978-3-642-18808-4},
       URL = {https://doi.org/10.1007/978-3-642-18808-4},
}

@misc{MPWN,
      title={{A transcendental non-Archimedean Calabi-Yau theorem with applications to the cscK problem}}, 
      author={Pietro Mesquita-Piccione and David {Witt Nystr\"om}},
      year={2025},
      note={arXiv:2509.09442},
      archivePrefix={arXiv},
      primaryClass={math.DG}
}

@article{Paul12,
 author = {Paul, Sean Timothy},
 title = {Hyperdiscriminant polytopes, {Chow} polytopes, and {Mabuchi} energy asymptotics},
 fjournal = {Annals of Mathematics. Second Series},
 journal = {Ann. Math. (2)},
 issn = {0003-486X},
 volume = {175},
 number = {1},
 pages = {255--296},
 year = {2012},
 language = {English},
 doi = {10.4007/annals.2012.175.1.7},
 zbMATH = {6014076},
 Zbl = {1243.14038}
}

@article{Kol98,
 author = {Ko{\l}odziej, S{\l}awomir},
 title = {The complex {Monge}-{Amp{\`e}re} equation},
 fjournal = {Acta Mathematica},
 journal = {Acta Math.},
 issn = {0001-5962},
 volume = {180},
 number = {1},
 pages = {69--117},
 year = {1998}
}

@article{BJ22b,
 author = {Boucksom, S{\'e}bastien and Jonsson, Mattias},
 title = {Global pluripotential theory over a trivially valued field},
 fjournal = {Annales de la Facult{\'e} des Sciences de Toulouse. Math{\'e}matiques. S{\'e}rie VI},
 journal = {Ann. Fac. Sci. Toulouse, Math. (6)},
 issn = {0240-2963},
 volume = {31},
 number = {3},
 pages = {647--836},
 year = {2022},
 language = {English}
}

@article {BBJ18,
    AUTHOR = {Berman, Robert J. and Boucksom, S\'{e}bastien and Jonsson,
              Mattias},
     TITLE = {A variational approach to the {Y}au-{T}ian-{D}onaldson
              conjecture},
   JOURNAL = {J. Amer. Math. Soc.},
  FJOURNAL = {Journal of the American Mathematical Society},
    VOLUME = {34},
      YEAR = {2021},
    NUMBER = {3},
     PAGES = {605--652},
      ISSN = {0894-0347},
   MRCLASS = {32Q20 (14E30 32P05 32Q26 32U35)},
  MRNUMBER = {4334189},
       DOI = {10.1090/jams/964},
       URL = {https://doi.org/10.1090/jams/964},
}

@Article{Ber19,
 Author = {Berman, Robert J.},
 Title = {From {Monge}-{Amp{\`e}re} equations to envelopes and geodesic rays in the zero temperature limit},
 FJournal = {Mathematische Zeitschrift},
 Journal = {Math. Z.},
 ISSN = {0025-5874},
 Volume = {291},
 Number = {1-2},
 Pages = {365--394},
 Year = {2019},
 Language = {English},
}

@article {Ber16,
    AUTHOR = {Berman, Robert J.},
     TITLE = {K-polystability of {${\Bbb Q}$}-{F}ano varieties admitting
              {K}\"{a}hler-{E}instein metrics},
   JOURNAL = {Invent. Math.},
  FJOURNAL = {Inventiones Mathematicae},
    VOLUME = {203},
      YEAR = {2016},
    NUMBER = {3},
     PAGES = {973--1025},
      ISSN = {0020-9910},
   MRCLASS = {53C55 (14J45 53C25)},
  MRNUMBER = {3461370},
MRREVIEWER = {Cristiano Spotti},
       DOI = {10.1007/s00222-015-0607-7},
       URL = {https://doi.org/10.1007/s00222-015-0607-7},
}

@article {BoJ18,
    AUTHOR = {Boucksom, S\'ebastien and Jonsson, Mattias},
     TITLE = "{A non-Archimedean approach to K-stability}",
      YEAR = {2018},
      note = {arXiv:1805.11160},
      label = {BoJ18},
}

@article {StSz,
    AUTHOR = {Stoppa, Jacopo and Sz\'{e}kelyhidi, G\'{a}bor},
     TITLE = {Relative {K}-stability of extremal metrics},
   JOURNAL = {J. Eur. Math. Soc. (JEMS)},
  FJOURNAL = {Journal of the European Mathematical Society (JEMS)},
    VOLUME = {13},
      YEAR = {2011},
    NUMBER = {4},
     PAGES = {899--909}
}

@misc{Pic24,
    title={{A non-Archimedean theory of complex spaces and the cscK problem}},
    author={P. {Mesquita-Piccione}},
    year={2024},
    note={arXiv:2409.06221 },
    archivePrefix={arXiv},
    primaryClass={math.AG}
}

@article {Zha20-conti,
    AUTHOR = {Zhang, Kewei},
     TITLE = {Continuity of delta invariants and twisted {K}\"{a}hler-{E}instein
              metrics},
   JOURNAL = {Adv. Math.},
  FJOURNAL = {Advances in Mathematics},
    VOLUME = {388},
      YEAR = {2021},
     PAGES = {Paper No. 107888, 25},
      ISSN = {0001-8708},
   MRCLASS = {53C55 (14C20 53C25)},
  MRNUMBER = {4288212},
       DOI = {10.1016/j.aim.2021.107888},
       URL = {https://doi.org/10.1016/j.aim.2021.107888},
}
\bibliographystyle{abbrv}

\noindent {\sc University of Maryland, College Park, USA}\\
{\tt tdarvas@umd.edu}\vspace{0.1in}\\
\noindent {\sc Beijing Normal University, Beijing, China}\\
{\tt kwzhang@bnu.edu.cn}\vspace{0.1in}

\end{document}